\documentclass[pdflatex,sn-mathphys-num]{sn-jnl}

\usepackage{comment}
\usepackage{graphicx}%
\usepackage{multirow}%
\usepackage{amsmath,amssymb,amsfonts}%
\usepackage{amsthm}%
\usepackage{mathrsfs}%
\usepackage[title]{appendix}%
\usepackage{xcolor}%
\usepackage{textcomp}%
\usepackage{manyfoot}%
\usepackage{booktabs}%
\usepackage{algorithm}%
\usepackage{algorithmicx}%
\usepackage{algpseudocode}%
\usepackage{listings}%
\usepackage{bm}
\usepackage{float}
\usepackage{threeparttable}

\usepackage{graphicx}
\usepackage[caption=false,font=footnotesize]{subfig}

\DeclareMathOperator*{\argmin}{\arg\!\min}

\newcommand{\real}{{\rm{I\hspace{-.75mm}R}}}

\newcommand{\Ebar}{\bar{E}}
\newcommand{\Fbar}{\bar{F}}
\newcommand{\Gbar}{\bar{G}}

\newcommand{\sigmahat}{\hat{\sigma}}

\newcommand{\Deltatilde}{\widetilde{\Delta}}

\newcommand{\BFe}{\bm{e}}

\newcommand{\BFq}{\bm{q}}

\newcommand{\BFs}{\bm{s}}

\newcommand{\BFx}{\bm{x}}
\newcommand{\BFy}{\bm{y}}

\newcommand{\BFD}{\bm{D}}
\newcommand{\BFE}{\bm{E}}

\newcommand{\BFG}{\bm{G}}
\newcommand{\BFH}{\bm{H}}

\newcommand{\BFX}{\bm{X}}
\newcommand{\BFS}{\bm{S}}

\newcommand{\BFEbar}{\bar{\BFE}}

\newcommand{\BFDbar}{\bar{\BFD}}
\newcommand{\BFGbar}{\bar{\BFG}}
\newcommand{\BFHbar}{\bar{\BFH}}

\newcommand{\mcA}{\mathcal{A}}
\newcommand{\mcB}{\mathcal{B}}

\newcommand{\mcD}{\mathcal{D}}

\newcommand{\mcF}{\mathcal{F}}
\newcommand{\mcG}{\mathcal{G}}

\newcommand{\mcJ}{\mathcal{J}}

\newcommand{\mcO}{\mathcal{O}}

\newcommand{\mcS}{\mathcal{S}}

\newcommand{\mcU}{\mathcal{U}}

\newcommand{\mcOtilde}{\widetilde{\mcO}}

\newcommand{\mbE}{\mathbb{E}}

\newcommand{\mbN}{\mathbb{N}}

\newcommand{\mbP}{\mathbb{P}}

\newcommand{\mbR}{\mathbb{R}}

\newcommand{\sfH}{\mathsf{H}}
\newcommand{\sfI}{\mathsf{I}}

\newtheorem{lemma}{Lemma}

\newtheorem{corollary}{Corollary}

\newtheorem{assumption}{Assumption}

\renewcommand{\thecondition}

\theoremstyle{thmstyleone}%
\newtheorem{theorem}{Theorem}
\theoremstyle{thmstyletwo}%
\newtheorem{remark}{Remark}%

\theoremstyle{thmstylethree}%
\newtheorem{definition}{Definition}%

\newcommand{\pre}{\textrm{pre}}
\newcommand{\s}{\textrm{s}}
\newcommand{\Deltapre}{\Delta_k^{\mathrm{pre}}}

\begin{document}

\title[Stochastic Trust Region with Quadratic Regularization]{Adaptive Regularization within Trust Region Methods for Stochastic Nonconvex Optimization}


\author[1]{\fnm{Yunsoo} \sur{Ha}}\email{yh2429@cornell.edu}

\author*[2]{\fnm{Sara} \sur{Shashaani}}\email{sshasha2@ncsu.edu}

\author[3]{\fnm{Quoc} \sur{Tran-Dinh}}\email{quoctd@email.unc.edu}





\abstract{We propose a stochastic nonconvex optimization algorithm that achieves almost sure $\mcOtilde(\epsilon^{-1.5})$ iteration complexity for problems with smooth objective functions and gradients only observable with noise. The mean-zero stochastic noise is decision-dependent and has unbounded support with subexponential tail,  allowing our framework to cover a broad class of problems. The improved almost sure iteration complexity is achieved with a new variant of the adaptive sampling trust-region optimization (ASTRO) augmented with an adaptively regularized local model, which we term Reg-ASTRO. Adaptive sampling ensures that the estimation  precision is aligned with a measure of stationarity, so that iterates closer to stationarity trigger higher accuracy requirement for sampling. A key analytical challenge arises because the trust-region radius and regularization are coupled and not determined prior to gradient estimation at each iteration. We further establish an almost sure $\mcOtilde(\epsilon^{-4.5})$ sample complexity for Reg-ASTRO, which  improves to $\mcOtilde(\epsilon^{-3.5})$ under stronger regularity conditions and use of common random numbers, substantially outperforming first-order methods in theory and numerical experiments.}





\maketitle

\section{Introduction}\label{sec:intro}
In this paper, we develop a first unconstrained stochastic trust region (TR) optimization framework for nonconvex objective functions that uses an adaptive regularization technique. TR methods offer a principled framework for nonconvex stochastic optimization through adaptive step-sizes and model-based exploitation of curvature. A typical TR method iteratively takes steps that optimize a quadratic approximation of the objective function within a radius from the incumbent that regulates the step-size. Adaptive regularization techniques attempt to regularize the approximation itself to ensure better directions are identified as a result of which the algorithm can attain improved complexity for the number of iterations.

We build on an existing algorithm called \texttt{ASTRO}~\cite{ha2025complexity}---adaptive sampling trust-region optimization---that uses a varying sample size for estimation of the function and gradients by forcing the stochastic error to be in the same order as the deterministic error which is governed by a measure of stationarity in the iterate. This mechanism saves a lot of sampling cost in the early and far-from-optimality iterations while increasing it based on the perceived proximity to the first-order critical point, inferred from the TR radius. We call the new algorithm \texttt{Reg-ASTRO}---the regularized \texttt{ASTRO}.

The adaptive regularization added to the TR subproblem will also determine how large the TR radius should be. This is one of the main distinctions between \texttt{Reg-ASTRO} and classical TR methods that have a geometric rate of increase or decrease for the TR radius depending on whether the current model in the current TR led to a trial point with sufficient progress.  

\subsection{Problem Statement}
We consider the following stochastic unconstrained nonconvex optimization problem 
\begin{equation} \label{eq:problem}
\operatorname*{min}_{\BFx \in \real^{d}} \ \left\{f(\BFx) =   \mbE [F(\BFx,\xi)] = \int_{\Xi} F(\BFx,\xi) \mbP(\mathrm{d}\xi)\right\},
\end{equation}
where $f:\real^{d} \rightarrow \real$ is a smooth function and $\xi$ is a random variable defined on the probability space $(\Xi,\mcF,\mbP)$. We assume that we have access to a stochastic first-order oracle that, for any $\BFx \in \real^d$, returns a noisy function value $F(\BFx,\xi)$ and a noisy gradient $\BFG(\BFx,\xi)$ satisfying $\mathbb{E}[F(\BFx,\xi)] = f(\BFx)$ and $\mathbb{E}[\BFG(\BFx,\xi)] = \nabla f(\BFx)$.
This means that at each iteration, one can obtain noisy evaluations of both the objective value and its gradient through the oracle, which provides a map $(\BFx,\xi)\mapsto (F(\BFx,\xi),\BFG(\BFx,\xi))$. Throughout the paper, we impose the following standing assumptions on the objective function $f$ and the stochastic oracle. (Throughout, $\|\cdot\|$ denotes Euclidean norm.)

\begin{assumption}[Smoothness]
\label{assum:smooth-f}
The function $f:\real^d \to \real$ is twice continuously differentiable and attains its minimum $f^*$ at some $\BFx^*$. 
Moreover, $\nabla f$ is $\kappa_{Lg}$-Lipschitz continuous, i.e.,
\[
\|\nabla f(\BFx_1)-\nabla f(\BFx_2)\|\le \kappa_{Lg}\|\BFx_1-\BFx_2\|,\quad \forall \BFx_1,\BFx_2\in\real^d,
\]
and $\nabla^2 f$ is $\kappa_{L\sfH}$-Lipschitz continuous, i.e.,
\[
\|\nabla^2 f(\BFx_1)-\nabla^2 f(\BFx_2)\|\le \kappa_{L\sfH}\|\BFx_1-\BFx_2\|,\quad \forall \BFx_1,\BFx_2\in\real^d.
\]
\end{assumption}

\begin{assumption}[Finite Variance] \label{assum:varcont} There is a $\bar\sigma \in (0,\infty)$ that satisfies 
\begin{align*}
    \sup_{\BFx \in \real^d}\left\{\sigma_F^2(\BFx) := \mathbb{E}\!\left[ (F(\BFx,\xi) - f(\BFx))^2 \right]\right\} 
    & \le \bar{\sigma}^2, \\ 
    \sup_{\BFx \in \real^d}\left\{\sigma_G^2(\BFx) := \mathbb{E}\!\left[ \|\BFG(\BFx,\xi) - \nabla f(\BFx)\|^2 \right]\right\} & \le \bar{\sigma}^2.
\end{align*}
\end{assumption}

A general approach to solve \eqref{eq:problem} is to apply an iterative algorithm that generates a sequence of random  iterates $\{\BFX_k\}_{k\ge 0}$ (using a capital letter to convey randomness), where each new iterate is constructed using information obtained from the stochastic oracle. As an example of such information, for $n$ independent and identically distributed (i.i.d.) samples $\{\xi_i\}_{i=1}^n$ at a point $\BFx \in \real^d$, 
we denote the function and gradient estimates by
\begin{equation}\label{eq:estimator}
    \Fbar(\BFx,n) := \frac{1}{n}\sum_{i=1}^n F(\BFx,\xi_i),
    \qquad
    \BFGbar(\BFx,n) := \frac{1}{n}\sum_{i=1}^n \BFG(\BFx,\xi_i),
\end{equation}
with the corresponding example of variance estimators
\begin{align*}
    \hat\sigma_F(\BFx,n) &:= \sqrt{\frac{1}{n-1}\sum_{i=1}^n \big(F(\BFx,\xi_i)-\Fbar(\BFx,n)\big)^2}, \\
    \hat\sigma_{\BFG}(\BFx,n) &:= \sqrt{\mathrm{Tr}\!\left[\frac{1}{n-1}\sum_{i=1}^n \big(\BFG(\BFx,\xi_i)-\BFGbar(\BFx,n)\big)\big(\BFG(\BFx,\xi_i)-\BFGbar(\BFx,n)\big)^\top\right]},
\end{align*}
where Tr$(\cdot)$ denotes the trace of a matrix. Throughout the paper, we will use the notation  
\begin{equation}
    \label{eq:sig-mx}
    \sigma_{\text{mx}}(\BFx,n) := \max\{\sigmahat_{F}(\BFx,n),\sigmahat_{\BFG}(\BFx,n),\sigma_0\}.
\end{equation}

In a stochastic setting, 
the measure of first-order complexity that is the first iteration that satisfies $\|\nabla f(\BFX_k)\|\le \epsilon$ for a pre-specified accuracy of $\epsilon>0$ is no longer a deterministic quantity.  We denote this random variable by $T_\epsilon$ and define it as $T_\epsilon=\min\{k:\ \|\nabla f(\BFX_k)\|\le \epsilon\}$. The complexity analysis would involve characterizing the expected value of $T_\epsilon$ or its probability of exceeding a presumed function of $\epsilon$. This is commonly known as the iteration-complexity. In the best case scenario we expect that a stochastic algorithm would perform as well as its deterministic counterpart in terms of iteration-complexity (but in a probabilistic sense). Still, since \emph{effort} in terms of total function evaluations at each iteration may not be fixed, a different complexity metric in the case of stochastic optimization is warranted. The \emph{sample complexity} terms the total number of random function evaluations (i.e., simulation runs) to reach $\epsilon$-stationarity; we denote this random variable by $W_\epsilon$. For a proposed algorithm with good finite-time guarantees, one needs both competitive iteration-complexity 
and sample complexity that incorporates cost of sampling and is hopefully not much worse than the iteration complexity.

\subsection{Regularization Methods}
The stochastic gradient method (SGD), which updates $\BFX_{k+1} = \BFX_k - \alpha_k \BFGbar(\BFX_k,n)$ with a predetermined step-size $\alpha_k$, represents the most basic use of the first-order stochastic oracle. However, for problems with ill-conditioned Hessians, the convergence of first-order methods can be extremely slow, which motivates the incorporation of second-order information.


Newton's method is a prototypical second-order algorithm. By constructing a quadratic approximation of the objective, it leverages curvature information to achieve faster convergence than first-order methods. However, its rapid convergence is guaranteed only locally, and the method remains highly sensitive to the initial guess. In fact, even with globalization strategies such as line search or trust regions, Newton’s method can converge as slowly as a basic first-order scheme~\cite{cartis2022evaluation}. These limitations have motivated the development of \emph{regularization methods}, which hinge on the $\mcO(\|\nabla f(\BFX_k)\|^{3/2})$ model reduction---larger than $\mcO(\|\nabla f(\BFX_k)\|^{2})$ reduction of first-order methods when $\|\nabla f(\BFX_k)\|<1$---achieved at each step to provide global convergence guarantees while retaining faster convergence rates of $\mcO(\epsilon^{-3/2})$ than that of $\mcO(\epsilon^{-2})$ in first-order approaches. 
A prominent example is the cubic regularization method~\cite{nesterov2006cubic}, which augments the quadratic model with a cubic term to ensure global convergence and control step lengths. 
This approach was later extended to the Adaptive Cubic Regularization (\texttt{ARC}) framework~\cite{cartis2011adaptive}, which updates the regularization parameter dynamically to enhance robustness in practice. Another line of work~\cite{doikov2024gradient,mishchenko2023regularized,gratton2025yet} uses a quadratic regularization, where the Newton step is modified by adding a multiple of the identity matrix to the Hessian. Specifically, the update takes the form
$\BFS_k = -\bigl(\nabla^2 f(\BFX_k) + \nu_k \sfI_d \bigr)^{-1}\nabla f(\BFX_k),$ with $\sfI_d$ being a $d\times d$ identity matrix and $\nu_k$ chosen adaptively, often following $\nu_k\propto \sqrt{\|\nabla f(\BFX_k)\|}$. 

Within the TR framework, this adaptive quadratic regularization perspective has been explored in different ways. For instance, the \texttt{TRACE} algorithm~\cite{curtis2017trace} incorporates regularization only implicitly, by adjusting the TR radius through subproblems involving a shifted Hessian. In contrast, the universal TR method~\cite{jiang2026beyond}~(\texttt{UTR}) makes quadratic regularization explicit by tying both the TR radius and the regularization term to $\sqrt{\|\nabla f(\BFX_k)\|}$. We will review \texttt{UTR} with more details in Section~\ref{sec:utr}.

\begin{remark}
In the quadratic regularization method for the deterministic setting, the function reduction is bounded below by $\mcO(\|\BFS_k\|^3)$ for successful iterations. However, there does not exist a constant $\theta>0$ such that $\|\BFS_k\| \ge \theta\|\nabla f(\BFX_k)\|$ for all successful steps. Consequently, a function reduction of order $\mcO(\epsilon^{3/2})$ cannot be ensured for every successful iteration. To address this issue, existing convergence analyses of quadratic regularization methods rely on the \emph{gradient contraction condition}, i.e., 
\(
\|\nabla f(\BFX_{k+1})\| \le \eta \|\nabla f(\BFX_k)\| \text{ for some } \eta \in (0,1),
\)
either explicitly or implicitly~\cite{gratton2025yet,jiang2026beyond}, i.e., a geometric decrease in gradient norms. In particular, successful iterations that do not yield sufficient function reduction yet satisfy this contraction property guarantee convergence. One can then limit the number of such iterations between consecutive $\mcO(\epsilon^{3/2})$-reducing steps to preserve 
the overall iteration complexity of $\mcO(\epsilon^{-3/2})$.


\end{remark}

All of the aforementioned developments pertain to deterministic optimization problems. In the stochastic setting, \texttt{ACR} algorithms have been extended to account for noisy function and gradient information~\cite{park2020stochasticACR,tripuraneni2018stochasticcubic,scheinberg2023first}. We are interested in staying within the constructs of a TR algorithm to leverage the recent developments in efficient adaptive sampling strategy of \texttt{ASTRO}. In this paper, devising an adaptively regularized  \texttt{ASTRO} has the main challenge of controlling the effect of noise while preserving the desired convergence and complexity guarantees. In deterministic settings, minimizing the regularized Taylor model yields a decrease of order $\|\BFS_k\|^3$. However, in the stochastic regime this decrease is perturbed by noise which must therefore be controlled at the same cubic order with a sampling strategy.
In principle, this requires the sample size to be chosen as a function of the step-size. However, because the step-size is determined only after solving the subproblem, this requirement is fundamentally at odds with the mechanics of the method. Consequently, existing stochastic extensions typically fix a target (worst-case) tolerance~$\epsilon$ to ensure
$\|\BFGbar(\BFX_k,n)-\nabla f(\BFX_k)\| \le \epsilon$, and adjust the sample size $n$ as a function of~$\epsilon$ for all $k$. Adopting the sampling strategy of \texttt{ASTRO}, in this paper will alleviate the burden (i.e., complexity) of forcing the estimation error to be small for all $k$.

\subsection{Summary of Insight and Results}
In this paper, we develop a new adaptive regularization and sampling algorithm (\texttt{Reg-ASTRO})~(Algorithm~\ref{alg:ASUTRO}) for stochastic nonconvex optimization. We prove that the adaptive quadratic regularization incorporated in the subproblem of a TR algorithm maintains the $T_{\epsilon}=\mcOtilde(\epsilon^{-1.5})$ in the almost sure sense for small enough $\epsilon$. Yet, given the cost of guaranteeing a sufficiently accurate model gradient that can lead to good selection of the TR radius $\Delta_k$ in each iteration, we establish a new adaptive sampling rule in each iteration of the algorithm in~\eqref{eq:as}. Using this practical sampling condition (summarized as a higher order of the TR radius in Table~\ref{tab:summary}), we show the improved sample complexity of $W_{\epsilon}=\mcOtilde(\epsilon^{-4.5})$ in the almost sure sense for small enough $\epsilon$ (Section~\ref{sec:complexity}). This rate may appear slower than what may be known as a stochastic first-order method \cite{arjevani2023lower}. However, a few points are noteworthy in the comparison against the  stochastic first-order methods. \emph{First}, the $\mcO(\epsilon^{-4})$ of the first-order methods is often a guarantee for the expected number of gradient evaluations~\cite{arjevani2023lower}. To achieve a high-probability-type result in that framework, one needs additional post-processing steps in the algorithms and stronger assumptions on the noise such as bounded support, heavy tails, or smoothness of sample paths~\cite{ghadimi2013stochastic, liu2023high,madden2024high}. \emph{Second}, the main advantage of using second-order methods with globalization, that is, the use of function values to make the accepted points adaptive, is that it helps the robustness of the algorithm, allowing it to perform well even starting from a poor initial guess, making it not dependent on the implicit knowledge of the Lipschitz constants and letting the step-size selection be adaptive, and enabling locally fast behavior once the algorithm reaches a close enough neighborhood. Notably, none of these gains are really visible when comparing this class of algorithms in terms of iteration complexity. 

Faster expected sample complexity rate of $\mcO(\epsilon^{-3})$ relies on mean-squared-smoothness of $\BFG(\cdot,\xi)$, that is, the existence of a constant $\kappa_{GmL}$  that for any $\BFx,\BFs\in\real^d$ satisfies
\[\mbE[\|\BFG(\BFx,\xi)-\BFG(\BFx+\BFs,\xi)\|^2]\le \kappa_{GmL}^2\|\BFs\|^2.\] In this paper, we do not make such an assumption.

Other sample-path-wise smoothness assumptions in \cite{levy2021storm+,fang2018spider} involve the existence of a  constant $\kappa_{GL}>0$
that for any $\BFx,\BFs\in\real^d$ and $\xi\in\Xi$ satisfies 
\begin{equation}\label{eq:lipG}
    \|\BFG(\BFx,\xi)-\BFG(\BFx+\BFs,\xi)\|\le \kappa_{GL}\|\BFs\|.
\end{equation} 
High-probability oracle complexity results in \cite{rinaldi2024stochastic} assume the existence of $\kappa_{FL}>0$ that for any $\BFx,\BFs\in\real^d$ and $\xi\in\Xi$ satisfies 
\begin{equation}\label{eq:lipF}
    |F(\BFx,\xi)-F(\BFx+\BFs,\xi)|\le \kappa_{FL}\|\BFs\|.
\end{equation}
To fully access the benefits of such structures in the sample path, one must fix the sample paths used to evaluate the function at the incumbent and the trial step, i.e., for iteration $k$ use the same fixed random numbers $\{\xi_{k,j},\ j=1,\cdots,n\}$ in \eqref{eq:estimator}, which is a notion we refer to as the Common Random Numbers (CRN). In fact,~\cite{ha2025complexity} proves that, under CRN, \texttt{ASTRO} can reduce its sampling rate requirement to achieve $W_{\epsilon}=\mcOtilde(\epsilon^{-4})$ with even a more relaxed assumption than Lipschitz sample paths: the existence of a constant $\kappa_{vL}>0$ that for any $\BFx,\BFs\in\real^d$ satisfies 
\begin{equation}\label{eq:lipv}
    |v(\BFx)-v(\BFx+\BFs)|\le \kappa_{vL}\|\BFs\|,
\end{equation} where $v(\BFx)=\mbE[(F(\BFx,\xi)-f(\BFx))^2]=\text{Var}(F(\BFx,\xi))$.
In contrast to \eqref{eq:lipF}, in \eqref{eq:lipv} we do not demand that each sample-path is Lipschitz continuous. We only require that at the population-level (expected value) the variability behaves continuously. For CRN gains in  \texttt{Reg-ASTRO}, we only need that the sample paths are differentiable (Assumption~\ref{assum:sample-path-c1}); Lipschitz sample-paths or Lipschitz variance is no longer needed (weaker compared to \texttt{ASTRO}). However, we require that the curvature of the model starts to vanish as $k\to \infty$ 
 (Assumption~\ref{assum:Hessian_bound_delta}). In such as scenario, we develop Algorithm~\ref{alg:ASUTRO-CRN} with simpler success criteria that can also use an inexact subproblem. We prove a better almost sure sample complexity bound of $W_{\epsilon}=\mcOtilde(\epsilon^{-3.5})$ for \texttt{Reg-ASTRO} in this setting (Section~\ref{sec:crn}).

\begin{remark}
    The deterministic TR algorithms classically do not need the TR radius $\Delta_k$ to converge to zero; in fact it is important for $\Delta_k$ to stay bounded away from zero so that near the criticality region, the subproblem constraint becomes non-binding and the steps become quasi-Newton steps. In deterministic derivative-free optimization settings (outside the scope of this paper), the bounded-away $\Delta_k$ is removed and in fact since $\Delta_k$ plays multiple roles (step-size regularization and model quality control) it will have to converge to zero eventually. To handle the issue of it becoming nonbinding close to the critical region, the acceptance criteria of a trial step is augmented by a secondary requirement, namely, that $\|\BFGbar(\BFX_k,n)\|\geq \mu \Delta_k$ for some $\mu\in(0,\infty)$. In turns out in the stochastic settings, similar to the derivative-free deterministic settings, since $\Delta_k$ plays the additional role of bounding the estimation error (which leads to the required model quality), one needs to apply the same format: $\Delta_k$ must converge to zero (in some probabilistic sense) and success must proceed two conditions simultaneously: a sufficient descent as well as $\|\BFGbar(\BFX_k,n)\|\geq \mu \Delta_k$. In this paper, we develop an analogous requirement that is coherent with the adaptive regularization mechanism. 
\end{remark}

Table~\ref{tab:summary} summarizes the important sample complexity results that we achieve in the almost sure (a.s.) sense. In this paper, we will not prove an analogue of using CRN with smooth sample paths (presented for the classic case in~\cite{ha2025complexity}) given that this assumption is quite strict in practice. But the interested reader may find the a.s. $W_{\epsilon}=\mcOtilde(\epsilon^{-2})$ for small enough $\epsilon$ in such a setting worthwhile.

\begin{table}[h]
\centering
\caption{Summary of sample sizes and complexity of \texttt{ASTRO}}\label{tab:synopsis}
\begin{tabular}{l c c c c}
\toprule
Case & CRN & Noise Assumptions & Sample Size Required & Complexity, $W_\varepsilon$\\
\midrule
\multirow{2}{*}{Classic}
& N & -- & $\tilde{\mcO}(\Delta_k^{-4})$ & $\tilde{\mcO}(\varepsilon^{-6})$ a.s. \\
& Y & Lipschitz Variance & $\tilde{\mcO}(\Delta_k^{-2})$ & $\tilde{\mcO}(\varepsilon^{-4})$ a.s. \\
\cmidrule(lr){1-5}
 \multirow{2}{*}{Regularized}
& N & -- & $\tilde{\mcO}(\Delta_k^{-6})$ & $\tilde{\mcO}(\varepsilon^{-4.5})$ a.s. \\
 & Y & Sample-path Differentiability & $\tilde{\mcO}(\Delta_k^{-4})$ & $\tilde{\mcO}(\varepsilon^{-3.5})$ a.s. \\
\bottomrule
\end{tabular}
\label{tab:summary}
\end{table}

\section{Preliminaries}\label{sec:prelim}

In this section, we introduce several definitions and assumptions that will be used throughout the analysis.  We begin with the notion of fully quadratic models, which characterizes the accuracy of local approximations. We start with a notations.

\subsection{Notation and Terminology}
We use bold font for vectors; for example, $\BFx=(x_1,x_2,\cdots,x_d)\in\real^d$ denotes a $d$-dimensional real-valued vector, and $\BFe^{i} \in \real^d$ for $i=1,\dots,d$ denotes the unit  vector in the $i$-th coordinate. We use capital letters for random scalars and vectors, calligraphic fonts for sets, and sans serif fonts for matrices. Our default norm operator $\|\cdot\|$ is an $L_2$ norm in the Euclidean space. 
We denote $\mcB(\BFx;\Delta)=\{\BFy\in\real^d:\|\BFy-\BFx\|_2\leq\Delta\}$ as the closed ball of radius $\Delta>0$ with center $\BFx$. For a sequence of sets $\{\mcA_n\}$, the set $\{\mcA_n \ \text{i.o.}\}$ denotes $\limsup_{n} \mcA_n$. 
We say $f(\BFx)=\mcO(g(\BFx))$ if there exist positive numbers $\varepsilon$ and $m$ such that $|f(\BFx)|\le mg(\BFx)$ for all $\BFx$ with $0 < |\BFx| < \varepsilon$. We say $f(\BFx) = \tilde{\mcO}(g(\BFx))$ if $f(\BFx)=\mcO(g(\BFx)\lambda(g(\BFx)))$, where $\lambda(\cdot)$ is a slowly varying function, i.e., $\lim_{t \to \infty} \frac{\lambda(xt)}{\lambda(x)} =1$ for each $x$. Common examples of slowly varying functions include $\log x, \log (\log x)$, and $(\log x)^r, r>0$.

\subsection{Key Definitions}
\begin{definition} (Fully Quadratic Models)
\label{defn:fullyquadratic} Given $\BFx \in \mbR^d$ and $\Delta > 0$, a function $m:\real^d \to \real$ is a fully quadratic model of $f$ (satisfying Assumption~\ref{assum:smooth-f}) on $\mcB(\BFx;\Delta)$ if 
there exist positive constants $\kappa_{\text{eg}}$ and $\kappa_{ef}$ dependent on $\kappa_{Lg}$ and $\kappa_{L\sfH}$ but independent of $\BFx$ and $\Delta$ such that $\forall \BFs\in\mcB(\BFx;\Delta)$
\begin{equation*}
\begin{split}
    \|\nabla^2 f(\BFx+\BFs) - \nabla^2 m(\BFs)\|&\le \kappa_{\text{e}\sfH}\Delta\\
    \|\nabla f(\BFx+\BFs) - \nabla m(\BFs)\|&\le \kappa_{\text{eg}}\Delta^2\\
    |f(\BFx+\BFs) - m(\BFs)| &\le \kappa_{ef}\Delta^3   
\end{split}
\end{equation*}
\end{definition}

To formalize the stochastic structure of the algorithm and the adaptive sampling scheme, we introduce the notions of filtration and stopping time.

\begin{definition} (Filtration and Stopping Time). A filtration $\{\mcF_{k}\}_{k \geq 1}$ over a probability space $(\Omega,\mathbb{P},\mcF)$ is defined as an increasing family of $\sigma$-algebras of $\mcF$, i.e., $\mcF_{k} \subset \mcF_{k+1} \subset \mcF$ for all $k$. We interpret $\mcF_{k}$ as ``all the information available at time $k$." A filtered space $(\Omega,\mathbb{P},\{\mcF_{k}\}_{k\geq1},\mcF)$ is a probability space equipped with a filtration.
A map $N:\Omega \rightarrow \{0,1,2,\dots,\infty\}$ is called a stopping time with respect to $\mcF_{k}$ if the event $\{N = n\} := \{\omega: N(\omega) = n\} \in \mcF_{k}$ for all $n \leq \infty$.
\end{definition}

For ease of exposition, in the remainder of the paper, we denote the function value stochastic error and estimation error, as well as gradient stochastic error  and estimation error as 
\begin{align*}
    E_{k,j} = F(\BFX_k,\xi_j) - f(\BFX_k),\ & \quad \Ebar_k(n) = \Fbar(\BFX_k,n) - f(\BFX_k); \\
    \BFE_{k,j}^g = \BFG(\BFX_k,\xi_j) - \nabla f(\BFX_k),\ & \quad \BFEbar_k^g(n) = \BFGbar(\BFX_k,n) - \nabla f(\BFX_k).
\end{align*}
The following assumption formalizes the stochastic noise structure and ensures suitable moment bounds for applying concentration inequalities in the subsequent analysis.

\begin{assumption}
    \label{assum:martingale}
    For a $\mcF_{k}$-measurable iterate $\BFX_k$, let $\mcF_{k,j}$ be the intermediate $\sigma$-algebras after each observation $j$ at iteration $k$ such that $\mcF_{k}\subseteq \mcF_{k,1}\subseteq \mcF_{k,2}\subseteq \cdots\subseteq \mcF_{k+1}$. The stochastic errors $E_{k,j}$ satisfy $\mbE[E_{k,j}|\mcF_{k,j-1}]=0$, and there exist constants  $\sigma_f^2>0$ and $b_f>0$ such that 
    \begin{equation}
    \mbE[|E_{k,j}|^m|\mcF_{k,j-1}]\leq\frac{m!}{2}b_f^{m-2}\sigma_f^2,\ \forall m=2,3,\cdots.\label{eq:subexp-f}
    \end{equation} 
Furthermore, let $[\BFE_{k}^g]_r$ be the $r$-th element of the stochastic gradient error. Then for all $i,j$, $[\BFE_{k,j}^g]_r$ satisfy, for any $r\in \{1,\dots,d\}$,  $\mbE[[\BFE_{k,j}^g]_r \ \vert\ \mcF_{k,j-1}]=0$. There also exist constants $\sigma_g^2>0$ and $b_g>0$ such that for a fixed $n$ and any $r\in \{1,\dots,d\}$, 
    \begin{equation}
        \mbE\left[|[\BFE_{k,j}^g]_r|^m|\mcF_{k,j-1}\right]\leq\frac{m!}{2}b_g^{m-2}\sigma_g^2,\ \forall m=2,3,\cdots.\label{eq:subexp-g}
    \end{equation} 
\end{assumption}

The inequalities in~\eqref{eq:subexp-f} and~\eqref{eq:subexp-g} impose restrictions on the growth of the moments of $E_{k,j}$ and $\BFE_{k,j}^g$, but these conditions are quite mild and are satisfied in many practical stochastic modeling settings. In particular, they hold whenever the underlying noise variables are light-tailed, meaning that their tails decay at least exponentially fast. A sufficient condition for this is that a real-valued random variable $X$ admits a moment generating function in a neighborhood of zero. Equivalently, there exists $\alpha>0$ such that for sufficiently large $x$, \[\mbP(X>x)\le \exp\{-\alpha x\}.\] 
This property is satisfied by most commonly used distributions in simulation and stochastic optimization, including the Gaussian, exponential, gamma, and any bounded random variable. In fact, any distribution with bounded support automatically satisfies the subexponential condition. It is also worth noting that requiring the existence of all moments of $X$ is strictly weaker than requiring the existence of a moment generating function in a neighborhood of zero. Thus, the subexponential-type bounds in~\eqref{eq:subexp-f} and~\eqref{eq:subexp-g} are compatible with a broad class of light-tailed distributions.


\subsection{Existing Results}
Here, we list three important results commonly invoked in the analysis. For background material on Bernstein-type inequalities for martingales, see~\cite{1999pen,2015fangraliu,1995van}.

\begin{lemma}[Lemma 4.1.1 in~\cite{nesterov2018lectures}]
\label{lem:smooth-bound}
Suppose that Assumption~\ref{assum:smooth-f} holds. 
Then, for all $\BFx, \BFy \in \mathbb{R}^n$, we have
\[
\bigl\|\nabla f(\BFy) - \nabla f(\BFx) - \nabla^2 f(\BFx)(\BFy - \BFx)\bigr\|
\le \tfrac{\kappa_{L\sfH}}{2}\,\|\BFy - \BFx\|^2.
\]
\end{lemma}

\begin{lemma}[Borel-Cantelli's First Lemma] 
\label{lem:borel-cantelli}
Let $(A_n)_{n\in \mbN}$ be a sequence of events defined on a probability space. If $\sum_{n=1}^\infty \mbP\left(A_n\right) < \infty$, then $\mbP\left(A_n \text{ i.o.}\right)=0$.     
\end{lemma}


\begin{lemma}[Bernstein's Inequality for Martingales]\label{lem:dependent-bernstein} Suppose that $(\xi_i,\mcF_i)_{\{i\ge 0\}}$ is a martingale difference sequence on a given probability space $(\Xi,\mcF,P)$, with $\xi_0=0$ and $\{\Xi,\emptyset\} = \mcF_0 \subseteq \mcF_1 \subseteq \mcF_2 \subseteq \cdots \subseteq \mcF_n \subseteq \mcF$ a sequence of increasing filtration and $\mbE\left[\xi_i \vert \mcF_{i-1}\right] = 0$. Suppose also that there exist $b>0$ and $\sigma^2>0$ such that $\mbE\left[ |\xi_i|^m \vert \mcF_{i-1} \right] \leq \frac{1}{2}\, m! \, b^{m-2} \sigma^2$ for $m=2,3,\cdots$.
Then for any $c>0$ and $n\in \mbN$, $$\mbP\left(\sum_{i=1}^n \xi_i \geq nc \right) \leq \exp\left\{- \frac{nc^2}{2(bc + \sigma^2)} \right\}.$$
\end{lemma}

\section{Reg-ASTRO}

We know that TR methods have attained optimal complexity rate of $\mcO(\epsilon^{-2})$ for nonconvex problems in both deterministic and stochastic settings under mild assumptions~\cite{ha2025complexity}. Any further improvement of this rate requires additional regularization schemes. In fact, as stated before, an  $\mcO(\epsilon^{-3/2})$ iteration complexity is possible in the deterministic setting~\cite{jiang2026beyond} and the main objective is to investigate how to adjust such an algorithm for a stochastic setting to not only obtain the same results but also without an overburdening sampling cost. Hence, we first describe how a TR algorithm could be enhanced with an adaptive quadratic regularization.

\subsection{Deterministic Trust Region With Adaptive Regularization}\label{sec:utr}
Recall that at iteration $k$ of a classical TR algorithm, a local quadratic model, typically constructed with Taylor expansion, is minimized within a TR ball constraint
\begin{align*}
    \min_{\BFs\in \real^d}\quad & m_k(\BFs):= f(\BFx_k) +\BFs^\intercal \nabla f(\BFx_k)+\frac{1}{2} \BFs^\intercal \nabla^2 f(\BFx_k) \BFs \\
    \text{s.t.}\quad & \|\BFs\|\leq \Delta_k.
\end{align*} The (approximate) solution to the above, $\BFs_k$ (also known as the trial step) is then evaluated as $f(\BFx_k+\BFs_k)$ and the progress observed is compared in a ratio with the progress predicted 
\[\rho_k=\frac{f(\BFx_k+\BFs_k)-f(\BFx_k)}{m_k(\BFs_k)-m_k(\boldsymbol{0})}\]
to make a decision about updating the TR and accepting the trial step. The step size control is done by increasing $\Delta_k$ (the TR radius) when successful and decreasing it when unsuccessful.

Jiang et al.~\cite{jiang2026beyond} provide a \emph{unified} TR framework (unified in the sense of accommodating both convex and nonconvex functions) by defining the subproblem as 
\begin{align}
    \min_{\BFs\in\real^d}\quad & m_k(\BFs)+ \frac{\tau_k}{2}\sqrt{\|\nabla f(\BFx_k)\|}\|\BFs\|^2 := f(\BFx_k) + \nabla f(\BFx_k)^\intercal \BFs \\
    &\quad\quad\quad\quad\quad\quad\quad\quad\quad\quad\quad\quad\quad\quad + \frac{1}{2}\BFs^\intercal (\nabla^2 f(\BFx_k)+\tau_k \sqrt{\|\nabla f(\BFx_k)\|} \sfI_d)\BFs \nonumber\\
    \text{s.t.}\quad & \|\BFs\|\le \Delta_k:=r_k \|\nabla f(\BFx_k)\|^{1/2}, \label{eq:det-quad-reg-subp}
\end{align} with two regularization hyperparameters $\tau_k$ and $r_k$. Intuitively, the regularization can tilt the direction of search roughly along $ (\nabla^2 f(\BFx_k) + \sqrt{\|\nabla f(\BFx_k)\|}\sfI_d)^{-1}\nabla f(\BFx_k)$ to find a solution that could sufficiently reduce the function relative to the reduction in the original (possibly nonconvex) quadratic model. Note, in lieu of updating $\Delta_k$ directly (as a bound on the step size), updating these hyperparameters applies control both on the step size and on the direction. 
The \texttt{UTR} (Universal Trust-Region) Algorithm proceeds by deeming an iteration successful if there is at least a monotone decrease, i.e., $f(\BFx_k+\BFs_k)<f(\BFx_k)$ and using another penalty hyperparameter $\vartheta_k$ if the subproblem solution $\BFs_k$ fails to achieve either of the success criteria below 
\begin{align}
    f(\BFx_k+\BFs_k) - f(\BFx_k) & \leq -\frac{\eta}{\vartheta_k}\|\nabla f(\BFx_k)\|^{3/2} \label{eq:sufficient-reduction} \\
    \|\nabla f(\BFx_k+\BFs_k)\| & \leq \zeta\|\nabla f(\BFx_k)\| \label{eq:gradient-contraction}
\end{align} for some $\eta\in(0,1)$ and $\zeta\in(0,1)$. The condition \eqref{eq:sufficient-reduction} resembles a sufficient reduction criterion common for typical TR algorithms, provided that the model reduction can be as good as a factor of $\|\nabla f(\BFx_k)\|^{3/2}$---this can be proven when solving~\eqref{eq:det-quad-reg-subp} exactly to fully exploit the curvature information. Even without the sufficient reduction of condition \eqref{eq:sufficient-reduction}, the iterate is still considered successful if condition \eqref{eq:gradient-contraction} holds, implying gradient contraction. 
The monotonic decrease requirement for success is an important requirement to find an upper bound for $\|\nabla f(\BFx_k)\|$ and for the number of successful iterations. 

\texttt{UTR} decreases the penalty hyperparameter $\vartheta_k$ after each iteration to reduce the regularization effect in the subproblem, while the effect on the TR radius may vary (it can increase or decrease). Iterations are all successful, as an inner loop keeps increasing $\vartheta_k$ until at least one of the two criteria is satisfied. 
The updating of $\tau_k$ and $r_k$ is more complicated and case by case using the knowledge of the smallest eigenvalue of $\nabla^2 f(\BFx_k)$ and for proving convergence to second-order stationarity (see Table~2 in~\cite{jiang2026beyond}).

\subsection{Algorithm and Main Ingredients for the Stochastic Setting}\label{sec:alg-ing}
In this paper, we seek a simplified algorithm with a single hyperparameter $\Lambda_k$ that can be utilized for the stochastic setting and convergence to first-order stationarity: an iterate is $\epsilon$-stationary if $\|f(\BFX_k)\|\le \epsilon$. 
This yields a simpler adaptive regularization that essentially decreases $\Lambda_k$ when successful and increases it otherwise. Importantly, we devise the algorithm not to rely on the knowledge of eigenvalues or the Lipschitz constant of the Hessian. Yet, conditions \eqref{eq:sufficient-reduction} and \eqref{eq:gradient-contraction} will need to be updated since we no longer have access to $\nabla f(\BFx_k)$. Moreover, we do not enforce monotone decrease as another success criteria as in \texttt{UTR}; instead we need to prove that it holds eventually (for large $k$) a.s. Therefore, at least in the early iterations our algorithm may be \emph{non-monotonic} given that it can accept a candidate point merely on the basis of gradient contraction (which is accompanied by an additional criticality condition to handle stochasticity). Another feature of our algorithm is that we do not cap $\Delta_k$ with a $\Delta_{\max}$ when solving the subproblem, although the \emph{preliminary} radius, as will be described below, will be capped by a $\Delta^{\pre}_{\max}$.

We list the new stochastic algorithm in Algorithm~\ref{alg:ASUTRO}. The inputs to the algorithm are an initial guess $\BFx_{0}$, initial and maximum TR radii $\Delta_{0}^\pre,\Delta^{\pre}_{\max}>0$, success threshold $\eta\in(1/4,1)$, step-size lower bound coefficient $\theta \in (0, 1)$, parameter update coefficients $\gamma_1>1>\gamma_2>0$, regularization parameter lower bound $ 0< \Lambda_{\min}$, 
gradient error scaler $c_g>0$, an adaptive sampling constant $\kappa_a>0$, sample size inflation sequence $\lambda_k = \mcO((\log k)^{1+\epsilon_{\lambda}})$, and diminishing tolerance sequence $\varepsilon_k = \Theta(k^{-2/3}) $. We will describe the algorithm steps and the role of these inputs below.

\begin{algorithm}[htp]  \scriptsize 
\caption{\texttt{Reg-ASTRO}: Adaptively Regularized \texttt{ASTRO} with Exact Subproblem}
\label{alg:ASUTRO}
\begin{algorithmic}[1]
\Require $\BFx_{0}, \Delta_{0}^\pre,\Delta_{\max}^\pre, \sigma_0, \Lambda_{\min}, \eta\in(\frac{1}{4},1), \mu>0, \theta \in (0, 1), \gamma_1>1>\gamma_2>0, \lambda_k >1$.

\State Estimate $\BFGbar_0^\pre$ using $\Delta_0^\pre$. Set $\Lambda_0=\max\left(\Lambda_{\min},\frac{\|\BFGbar_0^\pre\|}{16(\Delta_0^\pre)^2}\right)$ and $k=0$. \label{algstep:first-estimate-parameter}
\For{$k=0,1,2,\cdots$}
\State Evaluate $\Fbar_k=\Fbar(\BFX_k,N_k),\ \BFGbar_k=\BFGbar(\BFX_k,N_k)$ using \eqref{eq:sig-mx} and \eqref{eq:delta-tilde} with sample size
\begin{equation}
    N_k=\min\left\{n:\frac{\sigma_{\text{mx}}(\BFX_k,n)}{\sqrt{n}}\leq \frac{\kappa_a}{\sqrt{\lambda_k}} \Deltatilde_k^3(n)\right\}.\label{eq:as}
\end{equation} \label{algstep:evaluate}
\State Find a finite-difference approximation Hessian $\sfH_k\in\real^{d\times d}$ 
to construct the model
\begin{equation}
    M_k(\BFs) = \Fbar_k+\BFs^\intercal \BFGbar_k + \frac{1}{2}\BFs^\intercal \sfH_k \BFs.
\label{eq:model} 
\end{equation}
\label{ASUTRO:model}

\State Set $\Delta_k = \sqrt{\frac{\|\BFGbar_k\|}{16 \Lambda_k}}$ and the trial point $\BFX_k^\s=\BFX_k+\BFS_k$ where 
\begin{equation}
\BFS_k=\argmin\left\{M_k(\BFs)+\frac{1}{2}\Lambda_k \Delta_k\|\BFs\|^2:\ \BFs\in\real^d, \|\BFs\|\leq \Delta_k\right\}. \label{eq:subproblem}    
\end{equation}
 \label{ASUTRO:subproblem}

\State Evaluate $\Fbar_k^\s=\Fbar(\BFX_k^\s,N_k^\s)$ and $ \BFGbar_k^\s=\BFGbar(\BFX_k^\s,N_k^\s)$ with
\begin{equation}
    N_k^\s=\min\left\{n:\frac{\sigma_{\text{mx}}(\BFX_k^\s,n)}{\sqrt{n}}\leq \frac{\kappa_a}{\sqrt{\lambda_k}} \Delta_k^3\right\},\label{eq:as-candidate}
\end{equation}
and compute the success ratio 
\begin{equation}
    \rho_k = \frac{\Fbar_k - \Fbar_k^\s}{M_k(\boldsymbol{0})-M_k(\BFS_k)
}.\label{eq:success-ratio}
\end{equation}

\If{($\underbrace{\rho_k > \eta \text{ and } \|\BFS_k\| \ge  \theta \Delta_k}_{\text{Criteria 1}}$) or ($\underbrace{\|\BFGbar_k^\s\| \le \eta \|\BFGbar_k\| \text{ and } \Lambda_k > \max\{\mu\|\BFGbar_k\|,\lambda_k\}}_{\text{Criteria 2}}$)}

\label{ASUTRO:successful-iteration-condition}
\State Update $\BFX_{k+1} = \BFX_k^s$, $\Lambda_{k+1} = \max\left\{\gamma_2 \Lambda_k, \Lambda_{\min}\right\}$, and  $\BFGbar_{k+1}^\pre=\BFGbar_{k}^\s$. [Successful]

\Else
\State Update $\BFX_{k+1} = \BFX_k$, $\Lambda_{k+1} = \gamma_1 \Lambda_k$, and $\BFGbar_{k+1}^\pre=\BFGbar_{k}$. [Unsuccessful]

\EndIf
\State Set $\Delta_{k+1}^\pre = \min\left\{\sqrt{\frac{\|\BFGbar_{k+1}^\pre\|}{16\Lambda_{k+1}}},\Delta_{\max}^\pre\right\}$ and $k = k+1$.
\label{algstep:update-delta}
\EndFor
\end{algorithmic}
\end{algorithm}

Recall, the sample size $N_k$ for \eqref{eq:estimator} in \texttt{ASTRO}~\cite{ha2025complexity} is a random variable sequentially selected to retain the estimation error in the function values and gradients as commensurate with the \emph{improvement}, which would ensure (in some probabilistic sense)
\begin{align}\label{eq:classic-descent}
    \mcO(\Delta_k^2) = f(\BFX_{k})-f(\BFX_{k+1})&\ge \Fbar(\BFX_{k},N_k)-\Fbar(\BFX_{k+1},N_k) \\ & - |f(\BFX_{k})-\Fbar(\BFX_{k},N_k)-(f(\BFX_{k+1})-\Fbar(\BFX_{k+1},N_k))|.\nonumber
\end{align} In particular, $N_k$ is the smallest sample size satisfying a $\Delta_k$-prescribed precision 
\begin{equation}\label{eq:original-as}
    N_k = \min\left\{n:\ \underbrace{n^{-1/2} \sigma_{\text{mx}}(\BFX_k,n)}_{\text{estimation error}} \le \phi(k, \Delta_k)\right\},
\end{equation} where $\phi$ is a function that is determined based on the properties of the stochastic function value and gradient as well as the possibility of applying variance reduction. In the most rudimentary form of the classic \texttt{ASTRO}, $\phi(k, \Delta_k)=(\log k)^{-1}\Delta_k^2$, ensuring that the estimation error is only smaller than the required order of descent by a logarithmic factor of $k$. 

In the regularized setting, first observe that in lieu of \eqref{eq:classic-descent} we ultimately want $f(\BFX_{k})-f(\BFX_{k+1})=\mcO(\Delta_k^3)$---provided that we keep $\Delta_k$ and $\|\nabla f(\BFx_k)\|^{1/2}$ in tandem with a high probability. Hence the right-hand-side must look like $\phi(k, \Delta_k)=(\log k)^{-1}\Delta_k^3$. Before starting iterations, we use an initial sample size of $n_0=\left(\frac{\sigma_0}{\kappa_a(\Delta_0^\pre)^3}\right)^2$ to estimate the gradient $\BFGbar_0^\pre=\BFGbar(\BFx_0,n_0)$ at the initial solution and set the initial regularization parameter in Step~\ref{algstep:first-estimate-parameter} of Algorithm~\ref{alg:ASUTRO}. We will henceforth use this notation simplification, i.e., $\BFGbar_k^\pre,\BFGbar_k$ and apply the same rule for the function estimate. We will also use the notation $\Fbar_k^s,\BFGbar_k^\s$ for the function and gradient estimates at the trial point, i.e., $\Fbar(\BFX_k+\BFS_k,N_k^s)$ and $\BFGbar(\BFX_k+\BFS_k,N_k^s)$.

A difficulty in applying an adaptive sampling rule for \texttt{Reg-ASTRO} is to determine the sample size $N_k$ using $\Delta_k=\|\BFGbar_k\|^{1/2}(16\Lambda_k)^{-1/2}$, but as noted in Step~\ref{algstep:evaluate}, $\|\BFGbar_k\|$ itself is an estimate requiring knowledge of the sample size $N_k$ causing a circular dependence. If we use the intermediate gradient estimates after each new sample $\|\BFGbar(\BFX_k,n)\|$, then the right-hand-side of our sampling rule will also be random and hard to analyze. Therefore, we introduce a preliminary TR radius termed $\Deltapre$, which uses $\BFGbar_k^\pre$---the estimate of the current iterate at the end of the previous iteration. But since ultimately $\Delta_k$ will be needed to connect the estimation error and deterministic error, we instead require that the estimation error be bounded by an appropriate power $\alpha$ of $\Deltatilde_k(n)$ defined as 
\begin{equation}\label{eq:delta-tilde}
    \Deltatilde_k(n)=\max\left(\sqrt{\frac{\varepsilon_k}{\Lambda_k}},\min\left(\Deltapre, \sqrt{\frac{\|\BFGbar(\BFX_k,n)\|}{16 \Lambda_k}}-c_g\Deltapre\right)\right).
\end{equation}
for some constant $c_g>0$ and a diminishing deterministic tolerance $\varepsilon_k$. 

The sampling rule~\eqref{eq:as} is a critical contribution of this work. Suppose we use 
\begin{align}
    N_k=\min\left\{n:\frac{\sigma_{\text{mx}}(\BFX_k,n)}{\sqrt{n}}\leq \frac{\kappa_a}{\sqrt{\lambda_k}} \min\left(\Deltapre, \sqrt{\frac{\|\BFGbar(\BFX_k,n)\|}{16 \Lambda_k}}-c_g\Deltapre\right)^3\right\}\label{eq:incomplete-as}
\end{align}
instead of \eqref{eq:as}. Then using the proof arguments in Section~\ref{sec:analyze-estimates} and simplifying the notation $\BFEbar_k^g=\nabla f(\BFX_{k})-\BFGbar(\BFX_{k},N_k)$ and $\Ebar_k=f(\BFX_{k})-\Fbar(\BFX_{k},N_k)$, we prove that there is a constant $\kappa_{eg}>0$ that satisfies almost surely  for large enough $k$
\[\|\BFEbar_k^g\|\leq \kappa_{eg}(\Deltapre)^\alpha \text{ and }\|\BFEbar_k^g\|\leq \kappa_{eg}(\Delta_k - c_g\Deltapre)^\alpha\Rightarrow \|\BFEbar_k^g\|\leq \frac{\kappa_{eg}}{\left(c_g+1\right)^{\alpha}}\Delta_k^{\alpha}.\] The problem is that we may have a negative right-hand-side in \eqref{eq:incomplete-as}. In that case, we need to have a positive lower bound for the right-hand-side that grows at the desired rate. This right-hand-side lower bound will be determined by $\sqrt{\frac{\varepsilon_k}{\Lambda_k}}$. 
We make a crucial assumption that for sufficiently large $k$ we will have $\|\nabla f(\BFX_k)\|> \varepsilon_k$ where $\varepsilon_k = \Theta(k^{-2/3})$. Notice that this means that even if $\|\nabla f(\BFX_k)\|\to 0$, it will do so at a slower rate than $\varepsilon_k$, which implies that $T_{\varepsilon}$  will be larger than $\epsilon^{-3/2}$ in some probabilistic sense, which is a result one expects in the deterministic regularization setting~\cite{gratton2025yet}. 
\begin{assumption}
\label{assum:grad-dominates}
Let $\varepsilon_k=c_* k^{-2/3}$ with $c_*>0$. For every $\omega \in\Omega\backslash\Omega_0$, where $\Omega_0$ is the largest set of measure $0$, there exists $K(\omega)<\infty$ with
\begin{equation*}
\varepsilon_k < \|\nabla f_k(\omega)\|\text{ for all }k\ge K(\omega).
\end{equation*}
\end{assumption}
Why does this assumption make sense? Observe that if such a $K(\omega)$ does not exist, then the set $\mcJ(\omega)=\{k\in\mathbb N:\ \|\nabla f_k(\omega)\|\le \varepsilon_k\}$ is a nonempty infinite subsequence of iterates, and each $k\in \mcJ(\omega)$ is indeed the stopping time $T_{\varepsilon_k}(\omega)$. This then implies that $T_{\varepsilon_k}(\omega) \varepsilon_k^{3/2} \le c_*^{3/2}$ for all $k\in \mcJ(\omega)$, which is reminiscent, although weaker, of the optimal worst-case complexity result in the convex settings. The fact that this problem in nonconvex and stochastic suggests that such an optimal complexity rate is optimistic and not achievable, as evidenced even in the deterministic settings by~\cite{gratton2025yet}, for example. 
The resulting adaptive sampling rule \eqref{eq:as} will ultimately lead to eventually having both $\|\BFEbar_k^g\|\leq \kappa_{eg}\Delta_k^\alpha$ and $|\Ebar_k|\leq \kappa_{ef}\Delta_k^\alpha$ for given constants $\kappa_{eg},\kappa_{ef}>0$ and $\alpha\in\{0,1,2,3\}$ almost surely (see Theorem~\ref{thm:bounded-error}).


Note, we will need a stronger model gradient accuracy of $\mcO(\Delta_k^3)$ than the typical $\mcO(\Delta_k^2)$ accuracy of fully-quadratic models. This is because of the stochastic error causing additional complication in the analysis, as will be seen in use of \eqref{eq:true-reduction} to prove a.s. eventual monotonicity. We now provide more insight about this stronger requirement. With the exact solution to the subproblem~\eqref{eq:subproblem} we can only have such a lower-bound for the step-size:
\begin{equation}
\label{eq:stepsize_lb}
    \| \BFS_k \| \ge \frac{16 \Lambda_k \Delta_k^2}{\kappa_{\sfH} + \sqrt{\Lambda_k \|\BFGbar_k\|}}.
\end{equation}
To establish monotonicity for large iterations that are successful due to gradient contraction (Lemma~\ref{lem:monotone_decreasing}), we rely on this bound. However, in the presence of stochastic error, exploiting it to guarantee the desired reduction necessitates the stronger $\mcO(\Delta_k^3)$ model gradient accuracy.
On the other hand, in the special case of Section~\ref{sec:crn}, we show that the step-size is directly lower bounded by a factor of $\Delta_k$ and hence this stronger requirement on the model gradient accuracy is removed.

The Hessian approximation in Step~\ref{ASUTRO:model} of Algorithm~\ref{alg:ASUTRO} is a finite-difference approximation
using the gradient estimates at the current iterate and its perturbations. Specifically, define the gradient estimation errors $\BFEbar^{g,j}_{k} := \BFGbar(\BFX_k + \Delta_k \BFe_j, N_k^j) - \nabla f(\BFX_k + \Delta_k \BFe_j)$ for each $j \in \{1, \dots, d\}$, where the sample size $N_k^j$ is defined according to~\eqref{eq:as-candidate}, with $\BFX_k^\s$ replaced by $\BFX_k + \Delta_k \BFe_j$. Then for each $j \in \{1, \dots, d\}$, the $j$-th column of $\sfH_k$ can be decomposed into a true gradient difference and an estimation error term:
\begin{equation} \label{eq:hess-col-decomp}
\sfH_k \BFe_j = \frac{\nabla f(\BFX_k + \Delta_k \BFe_j) - \nabla f(\BFX_k)}{\Delta_k} + \frac{\BFEbar^{g,j}_{k} - \BFEbar^{g}_{k}}{\Delta_k}.
\end{equation} 
Consequently, the precision of the second-order approximation is inherently linked to the stochastic errors associated with the gradient estimates. Lemma~\ref{lem:fd-hessian} characterizes the $\mcO(\Delta_k)$ deviation of $\mathsf{H}_k$ from the true Hessian $\nabla^2 f(\BFX_k)$ in a probabilistic sense.

Lastly, as seen in Step~\ref{algstep:update-delta} of Algorithm~\ref{alg:ASUTRO}, success leads to contraction of the regularization coefficient $\Lambda_k$ and its expansion otherwise. Much like a key result in the classic TR that small enough $\Delta_k$ leads to success and therefore $\Delta_k$ remains bounded below by a function of $\|\BFGbar_k\|$, in this setting we will need similar things, i.e., success is warranted when $\Lambda_k$ becomes large enough leading to its \emph{upper bound} as  a function of $\|\BFGbar_k\|$(see Lemma~\ref{lem:reg-ub}). Our results give the insight that in the general setting, the ultimate bound on $\Lambda_k$ ends up being an increasing logarithmic sequence. This is to ensure that without knowledge of Lipschitz constants, the algorithm will eventually be successful. That being said, under the stronger conditions of Section~\ref{sec:crn}, we can establish a fixed and uniform upper bound for $\Lambda_k$ (see Lemma \ref{lem:reg-ub-crn}) resembling similar results in the deterministic setting.

\subsection{Almost Sure Eventually Accurate Models and Estimates}\label{sec:analyze-estimates}
This section provides the important control over the stochastic error by means of the adaptive sample size. In essence, our following results guarantee wp1 that once $k$ is large enough, the required model accuracy is assured. We obtain a stronger accuracy in this paper (relative to the typical $\mcO(\Delta_k^\alpha)$ as in \cite{ha2025complexity}) that will directly aid the analysis. Namely, we show that the model accuracy is eventually $\mcO(\sqrt{(\|\nabla f_k\|/\Lambda_k)^\alpha})$ a.s. Throughout, we use the superscript $0$ to denote quantities associated with the center point (incumbent), and the superscript $\mathrm{s}$ for those associated with the trial point.

\begin{lemma}\label{lem:gradient-asfinite} 
Suppose that Assumption~\ref{assum:martingale} holds. Then for iterates $\{\BFX_k\}$ generated by Algorithm~\ref{alg:ASUTRO} and any constants $\kappa_{eg}, \kappa_{ef} > 0$ and $\alpha \in [0,3]$, the following hold:
\begin{subequations} \label{eq:est-error-all}
\begin{align}
    \mbP\left(\|\BFEbar_k^{g,0}\| > \kappa_{eg}\max\left\{\Deltapre,\sqrt{\frac{\varepsilon_k}{\Lambda_k}}\right\}^{\alpha} \text{ i.o.}\right) &= 0, \label{eq:g-est-diff} \\
    \mbP\left(|\Ebar_k^{0}| > \kappa_{ef} \max\left\{\Deltapre,\sqrt{\frac{\varepsilon_k}{\Lambda_k}}\right\}^{\alpha} \text{ i.o.}\right) &= 0. \label{eq:f-est-diff}
\end{align}
\end{subequations}
\end{lemma}
\begin{proof}
We provide the proof for~\eqref{eq:g-est-diff}, as~\eqref{eq:f-est-diff} follows from analogous arguments. 
Fix $\kappa_{eg} > 0, \alpha \in [0,3]$.

\textbf{Case 1: $\Deltapre\ge \sqrt{\frac{\varepsilon_k}{\Lambda_k}}$.} 
From the sample size rule \eqref{eq:as}, we have $N_k \ge \lambda_k \nu_k$, where $\nu_k := \left(\frac{\sigma_0}{\kappa_a(\Deltapre)^3}\right)^2$. 
Using the union bound over $d$ dimensions and the property that $N_k$ is a stopping time satisfying $N_k \ge  \lambda_k \nu_k $, we bound the tail probability:
\begin{align}
    \mbP\left(\|\BFEbar^{g,0}_k\| > \kappa_{eg} (\Deltapre)^{\alpha} \mid \mcF_{k}\right)
    &\le \sum_{r=1}^d \mbP\left(\sup_{n \ge \lambda_k \nu_k} \left| \frac{1}{n} \sum_{j=1}^{n} [\BFE^{g,0}_{k,j}]_r \right| > \frac{\kappa_{eg}}{d}(\Deltapre)^{\alpha} \mid \mcF_{k}\right) \nonumber \\
    &\le \sum_{r=1}^d \sum_{n = \lambda_k \nu_k}^{\infty} \mbP\left( \left| \sum_{j=1}^{n} [\BFE^{g,0}_{k,j}]_r \right| > n \frac{\kappa_{eg}}{d}(\Deltapre)^{\alpha} \mid \mcF_{k}\right). \label{eq:union-over-n-case1}
\end{align}
Applying Lemma~\ref{lem:dependent-bernstein} to the martingale $S_n = \sum_{j=1}^{n} [\BFE^{g,0}_{k,j}]_r$, we have
\begin{equation} \label{eq:bernstein-detail}
    \mbP\left( |S_n| > n \frac{\kappa_{eg}}{d}(\Deltapre)^{\alpha} \mid \mcF_k \right) \le 2 \exp \left( - n\frac{ \frac{\kappa_{eg}^2}{d^2}(\Deltapre)^{2\alpha}}{ 2 \left( \sigma_g^2 + b_g \frac{\kappa_{eg}}{d}(\Deltapre)^{\alpha} \right) } \right).
\end{equation}
To simplify the exponent, we utilize the bound $\Deltapre \le \Delta_{\max}^\pre$. Define constant $c_1$ independent of $k$ as
\begin{equation} \label{eq:constant-c-def}
    c_1 := \frac{\kappa_{eg}^2}{ 2d^2 \sigma_g^2 + 2b_g d \kappa_{eg} (\Delta_{\max}^\pre)^{\alpha} } \le \frac{\kappa_{eg}^2}{ 2d^2 \sigma_g^2 + 2b_g d \kappa_{eg} (\Deltapre)^{\alpha} }.
\end{equation}
Substituting \eqref{eq:constant-c-def} into \eqref{eq:bernstein-detail} and summing the geometric series in \eqref{eq:union-over-n-case1} starting from $n = \lambda_k \nu_k$, we obtain
\begin{align}
    \mbP\left(\|\BFEbar^{g,0}_k\| > \kappa_{eg} (\Deltapre)^{\alpha} \mid \mcF_{k}\right) 
    &\le 2d \sum_{n = \lambda_k \nu_k }^{\infty} \exp\left( -n c_1 (\Deltapre)^{2\alpha} \right) \nonumber \\
    &\le 2d \frac{\exp(-\lambda_k \nu_k c_1 (\Deltapre)^{2\alpha})}{1 - \exp(-c_1 (\Deltapre)^{2\alpha})}. \label{eq:geometric-sum-case1}
\end{align}
Substituting the definition $\nu_k = \sigma_0^2 / (\kappa_a^2 (\Deltapre)^6)$ into the exponent of the numerator yields
\begin{equation} \label{eq:exponent-analysis}
    \lambda_k \nu_k c_1 (\Deltapre)^{2\alpha} 
    = \lambda_k \left( \frac{c_1 \sigma_0^2}{\kappa_a^2} \right) (\Deltapre)^{2\alpha-6} 
    \ge \left(\frac{c_1 \sigma_0^2 (\Delta_{\max}^\pre)^{2\alpha-6}}{\kappa_a^2} \right) \lambda_k,
\end{equation}
where the last inequality comes from $(\Deltapre)^{2\alpha-6} \ge (\Delta_{\max}^\pre)^{2\alpha-6}$ for any $\alpha \in [0,3]$.
Finally, since $\lambda_k = \mcO((\log k)^{1+\epsilon_{\lambda}})$, the probability in \eqref{eq:geometric-sum-case1} is summable over $k$.

\textbf{Case 2: $\Deltapre<\sqrt{\frac{\varepsilon_k}{\Lambda_k}}$.} The adaptive sampling rule~\eqref{eq:as} implies $N_k \ge \lambda_k \bar{\nu}_k$, where $\bar{\nu}_k := \sigma_0^2 \Lambda_k^3 / (\kappa_a^2 \varepsilon_k^3)$. Following the same geometric series argument as in Case 1 with the threshold $\kappa_{eg} (\varepsilon_k/\Lambda_k)^{\alpha/2}$, we obtain
\begin{equation} \label{eq:geometric-sum-case2}
    \mbP\left(\|\BFEbar^{g,0}_k\| > \kappa_{eg} \left(\frac{\varepsilon_k}{\Lambda_k}\right)^{\alpha/2} \mid \mcF_{k}\right) \le 2d \frac{\exp(-\lambda_k \bar{\nu}_k c_2 (\varepsilon_k/\Lambda_k)^{\alpha})}{1 - \exp(-c_2 (\varepsilon_k/\Lambda_k)^{\alpha})}.
\end{equation}
To ensure the constant $c_2$ is independent of $k$, we utilize the upper bound $\varepsilon_k/\Lambda_k \le \varepsilon_0/\Lambda_{\min}$. We define
\begin{equation} \label{eq:constant-c2-def}
    c_2 := \frac{\kappa_{eg}^2}{ 2d^2 \sigma_g^2 + 2b_g d \kappa_{eg} (\varepsilon_0/\Lambda_{\min})^{\alpha/2} } \le \frac{\kappa_{eg}^2}{ 2d^2 \sigma_g^2 + 2b_g d \kappa_{eg} (\varepsilon_k/\Lambda_k)^{\alpha/2} }.
\end{equation}
Substituting the definition of $\bar{\nu}_k$ into the exponent of the numerator in~\eqref{eq:geometric-sum-case2} yields
\begin{equation} \label{eq:exponent-analysis-case2}
    \lambda_k \bar{\nu}_k c_2 \left(\frac{\varepsilon_k}{\Lambda_k}\right)^{\alpha} = \lambda_k \left( \frac{c_2 \sigma_0^2}{\kappa_a^2} \right) \left(\frac{\varepsilon_k}{\Lambda_k}\right)^{\alpha-3} \ge \lambda_k \left( \frac{c_2 \sigma_0^2}{\kappa_a^2} \right) \left(\frac{\varepsilon_0}{\Lambda_{\min}}\right)^{\alpha-3},
\end{equation}
where the last inequality comes from $\alpha-3 \le 0$ and $\varepsilon_k/\Lambda_k \le \varepsilon_0/\Lambda_{\min}$. Finally, since $\lambda_k = \mcO((\log k)^{1+\epsilon_{\lambda}})$, the right-hand side of \eqref{eq:geometric-sum-case2} is summable over $k$. Applying Lemma~\ref{lem:borel-cantelli} in both cases concludes the proof. 
\end{proof}

\begin{remark}
\label{rmk:alpha_error}
The power $\alpha$ in \eqref{eq:est-error-all} represents the level of precision to which the estimation error is controlled relative to $\Deltapre$. This precision is a direct consequence of the specific power of $\Deltatilde_k$ used in the sampling stopping criterion. For instance, the cubic power of $\Deltatilde_k$ in~\eqref{eq:as} ensures that the conditions in \eqref{eq:est-error-all} hold for any $\alpha \in [0, 3]$. Accordingly, if the sampling rule were instead based on a quadratic power $\Deltatilde_k^2$, the resulting error would be governed by a corresponding range, namely $\alpha \in [0, 2]$. This link allows for the adjustment of the sampling rule's stringency based on the desired order of the estimation error.
\end{remark}

\begin{lemma} \label{lem:bounded-error}
Suppose Assumptions~\ref{assum:martingale} and \ref{assum:grad-dominates} hold and $\{\BFX_k\}$ is the sequence of iterates generated by Algorithm~\ref{alg:ASUTRO}.
Then, for any constants $\kappa_{eg} > 0$ and $\alpha \in [0,3]$, 
\begin{subequations}
\begin{align} 
    \mbP \left( \|\BFEbar_k^{g,0}\| > \kappa_{eg} \left(\frac{\|\nabla f_k^{0}\|}{\Lambda_k}\right)^{\alpha/2} \text{ i.o.} \right) &= 0, \label{eq:grad-error-fk} \\
    \mbP\left(|\Ebar_k^{0}| > \kappa_{ef} \left(\frac{\|\nabla f_k^{0}\|}{\Lambda_k}\right)^{\alpha/2} \text{ i.o.} \right) &= 0. \label{eq:func-error-fk}
\end{align}
\end{subequations}
\end{lemma}

\begin{proof}
We provide the proof for~\eqref{eq:grad-error-fk}, as~\eqref{eq:func-error-fk} follows from an analogous argument.
Let $\omega \in \Omega$ denote a realization of Algorithm~\ref{alg:ASUTRO} with a nonzero probability; for the remainder of this proof, we omit $\omega$ for simplicity. Fix any $\kappa_{eg}>0$ and $\alpha\in[0,3]$. The proof is partitioned into two cases depending on the relative scale of $\Deltapre$. In each case, we show that the tail probabilities of the gradient estimation error decay at a rate that ensures summability, thereby yielding the desired almost sure result.

\textbf{Case 1: $\Deltapre\ge \sqrt{\varepsilon_k/\Lambda_k}$.}
By Lemma~\ref{lem:gradient-asfinite}, we have $\|\BFGbar_k^0 - \nabla f_k^0\| \le \kappa_{eg} (\Deltapre)^2$ for all sufficiently large $k$. Choosing $\kappa_{eg}$ such that $\sqrt{\kappa_{eg}/\Lambda_{\min}} \le c_g$, where $c_g > 0$ is the constant given in~\eqref{eq:delta-tilde},  and utilizing the subadditivity of the square root, it follows that
\begin{align*}
    \sqrt{\frac{\|\BFGbar_k^0\|}{\Lambda_k}} - c_g \Deltapre
    &\le \sqrt{\frac{\|\nabla f_k^0\| + \|\BFGbar_k^0 - \nabla f_k^0\|}{\Lambda_k}} - c_g \Deltapre \\
    &\le \sqrt{\frac{\|\nabla f_k^0\|}{\Lambda_k}} + \left( \sqrt{\frac{\kappa_{eg}}{\Lambda_{\min}}} - c_g \right) \Deltapre \le \sqrt{\frac{\|\nabla f_k^0\|}{\Lambda_k}}.
\end{align*}
This implies that the adaptive sampling rule under Assumption~\ref{assum:grad-dominates} ensures a sample size satisfying $N_k^i \ge \lambda_k \bar{\nu}_k$, where $\bar{\nu}_k := \sigma_0^2 \Lambda_k^{3} / (\kappa_a^2 \|\nabla f_k^0\|^{3})$ for all sufficiently large $k$.
To analyze the tail probabilities, we partition the analysis into two regimes based on the relative magnitude of the gradient norm. 

In the regime where $\|\nabla f_k^0\| \ge \Lambda_k (\Deltapre)^2$, the concentration bound yields
\begin{equation*} \label{eq:ge-bound-regime1}
    \mbP\left(\|\BFEbar^{g,0}_k\| > \kappa_{eg} \left(\frac{\|\nabla f_k^0\|}{\Lambda_k}\right)^{\alpha/2} \ \middle\vert\ \mcF_{k}\right) \le 2d \frac{\exp(-\lambda_k \nu_k c_1 (\Deltapre)^{2\alpha})}{1 - \exp(-c_1 (\Deltapre)^{2\alpha})},
\end{equation*}
where $c_1$ and $\nu_k$ are the parameters defined as in the proof of Lemma~\ref{lem:gradient-asfinite}, and the summability of these probabilities follows from the same arguments therein.

Conversely, when $\|\nabla f_k^0\| < \Lambda_k (\Deltapre)^2$, we replace $(\Deltapre)^2$ with $\|\nabla f_k^0\|/\Lambda_k$ in the derivation of the concentration bound. This yields for sufficiently large $k$,
\begin{equation} \label{eq:geometric-sum-case2-cor}
    \mbP\left(\|\BFEbar^{g,0}_k\| > \kappa_{eg} \left(\frac{\|\nabla f_k^0\|}{\Lambda_k}\right)^{\alpha/2} \mid \mcF_{k}\right) \le 2d \frac{\exp(-\lambda_k \bar{\nu}_k c_2 (\|\nabla f_k^0\|/\Lambda_k)^{\alpha})}{1 - \exp(-c_2 (\|\nabla f_k^0\|/\Lambda_k)^{\alpha})},
\end{equation}
where 
\begin{equation*} 
    c_2 := \frac{\kappa_{eg}^2}{ 2d^2 \sigma_g^2 + 2b_g d \kappa_{eg} (\Deltapre)^{\alpha} } \le \frac{\kappa_{eg}^2}{ 2d^2 \sigma_g^2 + 2b_g d \kappa_{eg} (\|\nabla f_k^0\|/\Lambda_k)^{\alpha/2} }.
\end{equation*}
Substituting the definition of $\bar{\nu}_k$ into the exponent of the numerator in~\eqref{eq:geometric-sum-case2-cor} yields
\begin{equation*}
    \lambda_k \bar{\nu}_k c_2 \left(\frac{\|\nabla f_k^0\|}{\Lambda_k}\right)^{\alpha} 
    = \lambda_k \left( \frac{c_2 \sigma_0^2}{\kappa_a^2} \right) \left(\frac{\|\nabla f_k^0\|}{\Lambda_k}\right)^{\alpha-3} \ge \lambda_k \left( \frac{c_2 \sigma_0^2}{\kappa_a^2} \right) \left(\Deltapre\right)^{2\alpha-6}.
\end{equation*}
where the last inequality comes from $2\alpha-6 \le 0$ and $\|\nabla f_k^0\| < \Lambda_k (\Deltapre)^2$. Finally, since $\lambda_k = \mcO((\log k)^{1+\epsilon_{\lambda}})$ and $\Deltapre < \Delta_{\max}^\pre$, the right-hand side of~\eqref{eq:geometric-sum-case2-cor} is summable over $k$.

\textbf{Case 2: $\Deltapre < \sqrt{\frac{\varepsilon_k}{\Lambda_k}}$.}
By Assumption~\ref{assum:grad-dominates}, for sufficiently large $k$, we have $\|\nabla f_k^0\| \ge \varepsilon_k$, implying that 
\[
\mbP \left( \|\BFEbar_k^{g,0}\| > \kappa_{eg} \bigl(\Lambda_k^{-1}\|\nabla f_k^{0}\|\bigr)^{\alpha/2}\right) \le 
\mbP \left( \|\BFEbar_k^{g,0}\| > \kappa_{eg} \bigl(\Lambda_k^{-1}\varepsilon_k\bigr)^{\alpha/2}\right).
\]
The summability of the latter term is directly guaranteed by Lemma~\ref{lem:gradient-asfinite}. Consequently, applying Lemma~\ref{lem:borel-cantelli} in all cases concludes the proof for~\eqref{eq:grad-error-fk}. 
\end{proof}


Lemma~\ref{lem:bounded-error} establishes almost sure bounds on the stochastic gradient and function estimation errors at the center point ($i=0$). The following theorem extends this result to $\Delta_k$-based bounds and to both the center and trial points, i.e., $i \in \{0,s\}$.

\begin{theorem} \label{thm:bounded-error}
Suppose Assumptions~\ref{assum:martingale} and \ref{assum:grad-dominates} hold. For iterates $\{\BFX_k\}$ generated by Algorithm~\ref{alg:ASUTRO}, and any constants $\kappa_{eg}, \kappa_{ef} > 0$ and $\alpha \in [0,3]$, the following hold for any $i \in \{0,\text{s}\}$:
\begin{subequations} \label{eq:bounded-error-all-Delta}
\begin{align}
    \mbP \left( \|\BFEbar_k^{g,i}\| > \tilde\kappa_{eg} \Delta_k^{\alpha} \text{ i.o.} \right) &= 0, \label{eq:bounded-error-grad-delta} \\
    \mbP \left( |\Ebar_k^{i}| > \tilde\kappa_{ef} \Delta_k^{\alpha} \text{ i.o.} \right) &= 0, \label{eq:bounded-error-func-delta}
\end{align}
\end{subequations} where $\tilde\kappa_{eg}\le \kappa_{eg}(16/(1-\kappa_{eg}\Lambda_{\min}))^{\alpha/2}$ and $\tilde\kappa_{ef}\le \kappa_{ef}(16/(1-\kappa_{eg}\Lambda_{\min}))^{\alpha/2}$.
\end{theorem}
\begin{proof}
We provide the proof for~\eqref{eq:bounded-error-grad-delta}, as~\eqref{eq:bounded-error-func-delta} follows from an analogous argument. For the case $i=s$, the result is consistent with Case (A-0) of Theorem 3.4 in \cite{ha2025complexity}. For $i=0$, we first invoke Lemma \ref{lem:bounded-error}, which guarantees that for any arbitrary $\kappa_{eg} > 0$ and $\alpha \in [0,3]$, the estimation error satisfies $\|\BFEbar_k^{g,0}\| \le \kappa_{eg} (\|\nabla f_k^{0}\|/\Lambda_k)^{\alpha/2}$ almost surely for all sufficiently large $k$. 
Applying the reverse triangle inequality, we have for sufficiently large $k$,
\begin{equation} \label{eq:g-lowerbound-final}
\|\BFGbar_k^0\| \ge \|\nabla f_k^0\| - \|\BFEbar_k^{g,0}\| \ge \left(1 - \frac{\kappa_{eg}}{\Lambda_{\min}}\right) \|\nabla f_k^0\|.
\end{equation}
Using the trust-region update rule $\Delta_k^2 = \|\BFGbar_k^0\| / (16 \Lambda_k)$, we obtain
\begin{equation} \label{eq:grad-to-delta}
\frac{\|\nabla f_k^0\|}{\Lambda_k} \le \frac{\|\BFGbar_k^0\|}{\Lambda_k (1 - \kappa_{eg}/\Lambda_{\min})} = \frac{16 \Delta_k^2}{1 - \kappa_{eg}/\Lambda_{\min}}.
\end{equation}
Setting $c = 16 / (1 - \kappa_{eg}/\Lambda_{\min})$, it follows that
\begin{equation*}
\mbP \left( \|\BFEbar_k^{g,0}\| > \kappa_{eg} (c \Delta_k^2)^{\alpha/2} \right) \le \mbP \left( \|\BFEbar_k^{g,0}\| > \kappa_{eg} \bigl(\Lambda_k^{-1}\|\nabla f_k^{0}\|\bigr)^{\alpha/2} \right).
\end{equation*}
Thus, the summability of the right-hand side implies $\mbP ( \|\BFEbar_k^{g,0}\| > \tilde\kappa_{eg} \Delta_k^{\alpha} \text{ i.o.} ) = 0$. The result for $|\Ebar_k^0|$ follows by an analogous argument. In particular,
\begin{equation*}
    \mbP \left( |\Ebar_k^{0}| > \kappa_{ef} (c \Delta_k^2)^{\alpha/2} \right) \le \mbP \left( |\Ebar_k^{0}| > \kappa_{ef} \bigl(\Lambda_k^{-1}\|\nabla f_k^{0}\|\bigr)^{\alpha/2} \right),
\end{equation*}
which implies
\(
\mbP \left( |\Ebar_k^{0}| > \tilde\kappa_{ef} \Delta_k^{\alpha} \ \text{i.o.} \right) = 0.
\)
\end{proof}

We remark that the condition $\kappa_{eg} < \Lambda_{\min}^{-1}$ ensures that $1 - \kappa_{eg}\Lambda_{\min} > 0$ to ensure that $\tilde\kappa_{eg}>0$ in Theorem~\ref{thm:bounded-error}. We now quantify the accuracy of the Hessian approximation $\mathsf{H}_k$. 

\begin{lemma} \label{lem:fd-hessian}
Suppose Assumptions~\ref{assum:smooth-f}--\ref{assum:martingale} hold. 
Then, the forward-difference Hessian $\sfH_k$ constructed with step $\Delta_k$ satisfies
\begin{equation*}
\|\nabla^2 f(\BFX_k) - \sfH_k\| \le \sqrt{d} \left( \frac{\kappa_{L\sfH}}{2} \Delta_k + \frac{\max_{j} \|\BFEbar^{g,j}_{k} - \BFEbar^{g,0}_{k}\|}{\Delta_k} \right).
\end{equation*}
\end{lemma}

\begin{proof}
By the mean value theorem for vector fields, using \eqref{eq:hess-col-decomp} we obtain:
\begin{align*}
(\nabla^2 f(\BFX_k) - \sfH_k) \BFe_j &= \frac{1}{\Delta_k} \int_{0}^{\Delta_k} \left( \nabla^2 f(\BFX_k) - \nabla^2 f(\BFX_k + t \BFe_j) \right) \BFe_j \, dt - \frac{\BFEbar^{g,j}_{k} - \BFEbar^{g,0}_{k}}{\Delta_k}.
\end{align*}
Taking the Euclidean norm and applying the triangle inequality, the Lipschitz continuity of the Hessian (Assumption~\ref{assum:smooth-f}) implies:
\begin{align*}
\| (\nabla^2 f(\BFX_k) - \sfH_k) \BFe_j \| &\le \frac{1}{\Delta_k} \int_{0}^{\Delta_k} \kappa_{L\sfH} t \, dt + \frac{\|\BFEbar^{g,j}_{k} - \BFEbar^{g,0}_{k}\|}{\Delta_k} \\
&= \frac{\kappa_{L\sfH}}{2} \Delta_k + \frac{\|\BFEbar^{g,j}_{k} - \BFEbar^{g,0}_{k}\|}{\Delta_k}.
\end{align*}
Finally, using $\|\sfH_k\|_2 \le \|\sfH_k\|_F = (\sum_{j=1}^d \|\sfH_k \BFe_j\|_2^2)^{1/2}$, we have:
\begin{equation*}
\|\nabla^2 f(\BFX_k) - \sfH_k\| \le \sqrt{d} \max_j \| (\nabla^2 f(\BFX_k) - \sfH_k) \BFe_j \|,
\end{equation*}
which yields the desired bound upon substituting the column-wise result.
\end{proof}

\begin{remark} \label{rmk:autodiff}
In deterministic setting, automatic differentiation can be used to compute Hessian--vector products $\nabla^2 f(\BFx)\BFq$ at a cost that is only a small constant multiple of that for evaluating $f$; see Chapter 8.2 in~\cite{nocedal1999numerical}. Hence, a dense Hessian can be obtained by repeating this procedure for $\BFq=e_1,\dots,e_n$, yielding a cost proportional to $n$ times that of a single function evaluation. In stochastic simulation optimization, gradient information can in some cases be obtained via infinitesimal perturbation analysis (IPA), which corresponds to differentiating the sample performance with respect to the decision variables along a sample path; see~\cite{ford2022automatic}. From a second-order perspective, this framework can be extended to compute Hessian--vector products so that each such product can be obtained at a cost comparable to a small multiple of a single simulation run.
\end{remark}

\begin{corollary}\label{cor:hessian-error}
Suppose Assumptions~\ref{assum:smooth-f}--\ref{assum:grad-dominates} hold. Let $\{\BFX_k\}$ be the sequence of iterates generated by Algorithm~\ref{alg:ASUTRO}. Then, the forward-difference Hessian $\sfH_k$ satisfies
\begin{equation*}
    \mathbb{P}\left(\|\nabla^{2}f(\BFX_k) - \sfH_k\| > \kappa_{e\sfH} \Delta_k \text{ i.o.}\right) = 0,
\end{equation*}
where $\kappa_{e\sfH} := \sqrt{d} \left( \frac{\kappa_{L\sfH}}{2} + \tilde{\kappa}_{eg} \right)$ and $\tilde{\kappa}_{eg}$ is defined as in Theorem~\ref{thm:bounded-error}.
\end{corollary}

\begin{proof}
By following an argument analogous to the proof of Theorem~\ref{thm:bounded-error} (with $i=\s$), where the adaptive sample sizes $N_k^j$ are chosen according to~\eqref{eq:as-candidate} at each perturbed point $\BFX_k + \Delta_k \BFe_j$, we have for each $j \in \{1, \dots, d\}$,
\begin{equation} \label{eq:grad-error-purterb}
    \mathbb{P}\left( \|\BFEbar_k^{g,j}\| > \tilde\kappa_{eg} \Delta_k^2 \text{ i.o.} \right) = 0.
\end{equation}
Substituting this bound into Lemma~\ref{lem:fd-hessian} yields for all sufficiently large $k$,
\[
\|\nabla^2 f(\BFX_k) - \sfH_k\| \le \sqrt{d} \left( \frac{\kappa_{L\sfH}}{2} \Delta_k + \frac{1}{\Delta_k} \cdot \tilde\kappa_{eg} \Delta_k^2 \right) = \kappa_{e\sfH} \Delta_k,
\]
which completes the proof.
\end{proof}

With Theorem~\ref{thm:bounded-error}
and Corollary~\ref{cor:hessian-error}, we almost surely achieve similar model accuracy guarantee as in fully quadratic models (Definition~\ref{defn:fullyquadratic}) for large enough iterations. We also provide a stronger guarantee of $\mcO(\Delta_k^3)$ on the model gradient, which (as discussed in Section~\ref{sec:alg-ing}) is needed because of stochasticity to ensure eventual monotonicity. Moreover, the stronger gradient error bound based on the true gradient directly in Lemma~\ref{lem:bounded-error} will simplify the complexity analysis, specifically, in finding an upper bound for the true gradient (see Lemma~\ref{lem:monotone_decreasing}).

\section{Complexity Results}\label{sec:complexity}

In this section, we provide complexity analysis for Algorithm~\ref{alg:ASUTRO}. We first assume a uniform bound on the model Hessian to control the curvature of the quadratic subproblem. Based on this, we derive the optimality conditions and a lower bound for the predicted model reduction.

\begin{assumption}[Bounded Hessian in Norm] \label{assum:lipschtizmodelHessian}
    In Algorithm~\ref{alg:ASUTRO}, the model Hessian $\sfH_k$ satisfies  $\|\sfH_k\|\le \kappa_{\sfH}$ for all $k$ and some $\kappa_{\sfH}\in(0,\infty)$ almost surely.
\end{assumption}
This assumption would be automatically satisfied for large $k$ using the fact that if $f$ is twice-continuously differentiable and $\nabla f$ is $\kappa_{Lg}$-Lipschitz (Assumption~\ref{assum:smooth-f}), then $\|\nabla^2 f\|\le \kappa_{Lg}$, and therefore by  Corollary~\ref{cor:hessian-error} we eventually have $\|\sfH_k\|\le \|\nabla^2 f(\BFX_k)\|+\|\nabla^2 f(\BFX_k)-\sfH_k\|\le\kappa_{Lg}+\kappa_{e\sfH}$.

\begin{lemma}
\label{lem:exactsln}
    For each $k \in \mathbb{N}$, $\BFS_k$ is a global solution of \eqref{ASUTRO:subproblem} if and only if there exists a dual multiplier $\ell_k \ge 0$ satisfying
\begin{align}
    \ell_k\big(\|\BFS_k \| - \Delta_k\big) &= 0, \label{eq:opt-2}\\
    (\sfH_k + \sqrt{\Lambda_k \|\BFGbar_k^0\|}\,\sfI_d + \ell_k \sfI_d)\BFS_k &= -\BFGbar_k^0, \label{eq:opt-3}\\
    \sfH_k + \sqrt{\Lambda_k \|\BFGbar_k^0\|}\,\sfI_d + \ell_k \sfI_d &\succeq 0. \label{eq:opt-4}
\end{align}
\end{lemma}
\begin{proof}
    The results are directly obtained from Theorem 4.1 in \cite{nocedal1999numerical}.
\end{proof}

\begin{lemma} \label{lem:model-reduction}
    Let $(\BFS_k,\ell_k)$ satisfy the optimal conditions~\eqref{eq:opt-2}-\eqref{eq:opt-4}. Then, for all $k \in \mathbb{N}$,
    \begin{equation} 
        M_k(\boldsymbol{0})-M_k(\BFS_k) \ge 
        2 \Lambda_k \Delta_k\|\BFS_k\|^2.
    \end{equation}
\end{lemma}

\begin{proof}
Let $\omega \in \Omega$ denote a realization of Algorithm~\ref{alg:ASUTRO} with a nonzero probability. For the remainder of this proof, we omit the $\omega$ for simplicity.
From \eqref{eq:model}, we obtain
\begin{equation*}
    \begin{split}
        M_k(\boldsymbol{0})-M_k(\BFS_k) & = -(\BFS_k^\intercal\BFGbar_k^0+\frac{1}{2}\BFS_k^\intercal \sfH_k \BFS_k) \\  
        & = \BFS_k^\intercal(\sfH_k +  \sqrt{\Lambda_k \|\BFGbar_k^0\|} \sfI_d +\ell_k \sfI_d)\BFS_k - \frac{1}{2}\BFS_k^\intercal \sfH_k \BFS_k \\
        & =  \frac{1}{2}(\sqrt{\Lambda_k  \|\BFGbar_k^0\|} + \ell_k) \|\BFS_k\|^2 + \frac{1}{2}\BFS_k^\intercal (\sfH_k +   \sqrt{\Lambda_k \|\BFGbar_k^0\|} \sfI_d +\ell_k \sfI_d ) \BFS_k \\
        & \ge \frac{1}{2}(  \sqrt{\Lambda_k \|\BFGbar_k^0\|} +\ell_k) \|\BFS_k\|^2 \ge 2 \Lambda_k \Delta_k\|\BFS_k\|^2,
    \end{split}
\end{equation*}
where the first inequality uses~\eqref{eq:opt-4}, while the last inequality uses $\|\BFGbar_k^0\| = 16 \Lambda_k \Delta_k^2$.
\end{proof}

We now establish that the regularization parameter $\Lambda_k$ remains controlled and does not grow arbitrarily large. Since the update of $\Lambda_k$ is governed by the success of each iteration, we show that once $\Lambda_k$ exceeds a sufficiently large threshold, the algorithm consistently generates successful steps. To facilitate this analysis, we partition the iterations into two representative index sets as follows,
\begin{align*}
\mcD &:= \{k \in \mathbb{N} : (\rho_k \ge \eta) \bigcap (\|\BFS_k\| \ge \theta \Delta_k)\}, \\
\mcG &:= \{k \in \mathbb{N} : (\|\BFGbar_k^\s\| \le \eta \|\BFGbar_k^0\|) \bigcap (\Lambda_k \ge \max\{\mu \|\BFGbar_k^0\|, \lambda_k\})\}.
\end{align*}
The set $\mcD$ corresponds to successful iterations due to Criteria 1 in Algorithm~\ref{alg:ASUTRO} (counterpart of Condition~\eqref{eq:sufficient-reduction} in the deterministic setting); iterations in $\mcD$ have relatively large step-sizes. In contrast, $\mcG$ represents successful iterations due to Criteria 2 (counterpart of Condition~\eqref{eq:gradient-contraction} in the deterministic setting); iterations in $\mcG$ have gradient magnitudes effectively reduced through interior steps.

\begin{lemma} \label{lem:reg-ub}
    Suppose that Assumptions~\ref{assum:smooth-f}--\ref{assum:lipschtizmodelHessian} hold. For iterates $\{\BFX_k\}$ generated by Algorithm~\ref{alg:ASUTRO}, there exist constants $\Lambda_\mcD$ and $\Lambda_{\mcG}$ such that
    \begin{equation*}
        \mathbb{P} \left( \Lambda_k > \gamma_1 \max\{\Lambda_\mcD,\Lambda_\mcG,\mu\|\BFGbar_k^0\|, \lambda_k\}  \ \text{i.o.} \right) = 0.
    \end{equation*}
\end{lemma}

\begin{proof}
    Let $\omega \in \Omega$ denote a realization of Algorithm~\ref{alg:ASUTRO} with a nonzero probability; we omit $\omega$ for simplicity. The analysis is partitioned into two cases based on the value of the dual multiplier: $\ell_k = 0$ and $\ell_k > 0$.
    
    \textbf{Case 1}: $\ell_k > 0$. In this case, \eqref{eq:opt-2} yields the active constraint $\|\mathbf{S}_k\| = \Delta_k$. By the third-order Taylor expansion and the Lipschitz continuity of the Hessian (Assumption~\ref{assum:smooth-f}), the function value at the candidate point satisfies
    \begin{equation} \label{eq:function-reduction-actual}
        \underbrace{\Fbar_k^s-\Ebar_k^s}_{f_k^\s} \le \underbrace{\Fbar_k^0-\Ebar_k^0}_{f_k^0} + \BFS_k^\top \nabla f_k^0 + \frac{1}{2}\BFS_k^\top \nabla^2 f_k^0 \BFS_k + \frac{\kappa_{L\sfH}}{6}\|\BFS_k\|^3.
    \end{equation} 
    From~\eqref{eq:model} and~\eqref{eq:function-reduction-actual}, we obtain for sufficiently large $k$,
    \begin{align*}
        |\Fbar_k^s-M_k(\BFS_k)| &\le \| \nabla f_k^0 - \BFGbar_k^0 \| \|\BFS_k\| + \frac{1}{2} \|\nabla^2 f_k^0 - \sfH_k\| \|\BFS_k\|^2 + \frac{\kappa_{L\sfH}}{6} \|\BFS_k\|^3 + |\Ebar_k^{\s} - \Ebar_k^{0}| \\
        &\le \tilde\kappa_{eg} \Delta_k^3 + \frac{\kappa_{e\sfH}}{2} \Delta_k^3 + \frac{\kappa_{L\sfH}}{6} \Delta_k^3 + 2 \tilde\kappa_{ef} \Delta_k^3 
        = \left(\tilde\kappa_{eg} + \frac{\kappa_{e\sfH}}{2} + \frac{\kappa_{L\sfH}}{6}+2\tilde\kappa_{ef}\right)\Delta_k^3,
    \end{align*}
    where the estimation error bounds follow from Theorem~\ref{thm:bounded-error} and Corollary~\ref{cor:hessian-error}. 
    From Lemma~\ref{lem:model-reduction} and the fact that $\|\BFS_k\| = \Delta_k$, the success ratio satisfies 
    \begin{equation*}
    \begin{split}    
        |1 - \rho_k| = \frac{|\Fbar_k^s-M_k(\BFS_k)|}{|M_k(\boldsymbol{0})-M_k(\BFS_k)|} 
        &\le \frac{\left(\tilde\kappa_{eg} + \frac{\kappa_{e\sfH}}{2} + \frac{\kappa_{L\sfH}}{6}+2\tilde\kappa_{ef}\right) \Delta_k^3}{2 \Lambda_k \Delta_k^3} \\ 
        &\le \frac{\left(\tilde\kappa_{eg} + \frac{\kappa_{e\sfH}}{2} + \frac{\kappa_{L\sfH}}{6}+2\tilde\kappa_{ef}\right)}{2 \Lambda_{\mcD}}.
    \end{split}
    \end{equation*}
    When $\Lambda_{\mcD}=\left(\tilde\kappa_{eg} + \frac{\kappa_{e\sfH}}{2} + \frac{\kappa_{L\sfH}}{6}+2\tilde\kappa_{ef}\right)(2-2\eta)^{-1}$, we ensure that $|1 - \rho_k| \le 1 - \eta$, which implies $\rho_k \ge \eta$. Thus, $k \in \mcD$ for sufficiently large $k$.

    \textbf{Case 2}: $\ell_k = 0.$ For sufficiently large $k$, we use the triangle inequality to bound $\|\BFGbar_k^\s \|$ as follows:
    \begin{align*}
    A_k &:= \|\BFGbar_k^\s - \nabla f_k^\s\| \le \tilde\kappa_{eg} \Delta_k^2, \\
    B_k &:= \|\nabla f_k^\s - (\nabla f_k^0 + \nabla^2 f_k^0 \BFS_k)\| \le \frac{\kappa_{L\sfH}}{2}\|\BFS_k\|^2, \\
    C_k &:= \|(\nabla f_k^0 + \nabla^2 f_k^0 \BFS_k) - (\BFGbar_k^0 + \sfH_k \BFS_k)\| \le \tilde\kappa_{eg}\Delta_k^2 + \kappa_{e\sfH} \Delta_k^{2}.
    \end{align*}
    Here, the bound for \(A_k\) follows from Theorem~\ref{thm:bounded-error}, the bound for \(B_k\) from Lemma~\ref{lem:smooth-bound}, and the bound for \(C_k\) from Theorem~\ref{thm:bounded-error} and Corollary~\ref{cor:hessian-error}.
    The triangle inequality with~\eqref{eq:opt-3} yields
    \begin{align*}
        \|\BFGbar_k^\s \| 
        \le A_k + B_k + C_k + \| \BFGbar_{k}^0 + \sfH_k \BFS_k \| 
        \le \left(2\tilde\kappa_{eg} + \frac{\kappa_{L\sfH}}{2} + \kappa_{e\sfH} \right) \Delta_k^2 + \sqrt{\Lambda_k\|\BFGbar_k^0\|}\Delta_k.
    \end{align*}
    Substituting $\Delta_k^2=\|\Gbar_k^0\|/(16\Lambda_k)$ yields
    \begin{equation} \label{eq:grad-contr-refined}
        \|\BFGbar_k^{\mathrm{s}}\| \le \left(\underbrace{\left( \frac{\tilde\kappa_{eg}}{8} + \frac{\kappa_{L\sfH}}{32} + \frac{\kappa_{e\sfH}}{16}\right)}_{:=\kappa_{g}}\cdot \frac{1}{\Lambda_k} + \frac{1}{4} \right) \|\BFGbar_k^0\|.
    \end{equation}
    For any target contraction factor $\eta \in (1/4, 1)$, if $\Lambda_k \ge \kappa_g/(\eta-\frac{1}{4})$, $\|\BFGbar_k^{\mathrm{s}}\| \le \eta \|\BFGbar_k^0\|$ holds for all sufficiently large $k$. Consequently, by defining
    \[
    \Lambda_k \ge \max\left\{ \underbrace{\frac{\kappa_g}{\eta-\frac{1}{4}}}_{:=\Lambda_{\mcG}}, \mu\|\BFGbar_k^0\|, \lambda_k \right\}
    \]
    it follows that $k \in \mathcal{G}$ for all sufficiently large $k$. 
    
    In conclusion, when 
    \(
    \Lambda_k \ge \max\{\Lambda_{\mcD}, \Lambda_{\mcG}, \mu\|\BFGbar_k^0\|, \lambda_k\},
    \)
    it follows that, for all sufficiently large $k$, either $k \in \mcD$ or $k \in \mcG$. 
    Hence, $\Lambda_k$ cannot exceed $\gamma_1 \max\{\Lambda_{\mcD}, \Lambda_{\mcG}, \mu\|\BFGbar_k^0\|, \lambda_k\}$ infinitely often, which completes the proof.
\end{proof}

An interesting observation from the proof of Lemma~\ref{lem:reg-ub} is that the successful steps that are on the boundary (due to the model being nonconvex or without a global minimum inside the TR) eventually turn out to produce sufficient reduction, whereas the successful interior steps (due to the model being strongly 
convex) eventually turn out to produce gradient contraction even though their reduction in the function value is not guaranteed. We next establish an upper bound on the true gradient norm along the trajectory, showing that it is controlled by the logarithmically growing scale $\lambda_k^{1/3}$.

\begin{lemma} \label{lem:grad-ub-io}
Suppose Assumptions~\ref{assum:smooth-f}--\ref{assum:lipschtizmodelHessian} hold. If objective function $f(\cdot)$ is bounded above by $f^u$, then the iterates $\{\BFX_k\}$ generated by Algorithm~\ref{alg:ASUTRO} ensures the existence of a finite positive constant $\kappa_u$ such that
\begin{equation*}
\mathbb{P} \left( \|\nabla f_k^0\| > \kappa_u \lambda_k^{1/3} \ \text{i.o.} \right) = 0.
\end{equation*}
\end{lemma}

\begin{proof}
    Let $\omega \in \Omega$ denote a realization of Algorithm~\ref{alg:ASUTRO} with a nonzero probability. We omit $\omega$ for simplicity. 
    From Lemma~\ref{lem:reg-ub}, we have $\Lambda_k \le \gamma_1 \max\{\Lambda_\mcD,\Lambda_\mcG,\mu\|\BFGbar_k^0\|, \lambda_k\} =: \Lambda_k^{\max}$ for all sufficiently large $k$. Since $\Lambda_{\mcD}$ and $\Lambda_{\mcG}$ are fixed positive constants and $\lambda_k \to \infty$, for all sufficiently large $k$,
    \(
    \lambda_k \ge \max\{\Lambda_{\mcD},\Lambda_{\mcG}\}.
    \)
    Hence, 
    \(
    \Lambda_k^{\max} = \max\{\mu\|\BFGbar_k^0\|,\lambda_k\}
    \)
    for all sufficiently large $k$. Therefore, it suffices to consider the following two cases: (i) $\Lambda_k^{\max} = \mu \|\BFGbar_k^0\|$ and (ii) $\Lambda_k^{\max} = \lambda_k$.
    
    \textbf{Case 1: $\Lambda_k^{\max} = \mu \|\BFGbar_k^0\|$.} 
    First, consider the case $k \in \mcD$. 
    By the definition of the success ratio $\rho_k \ge \eta$ and the model reduction bound in Lemma~\ref{lem:model-reduction}, the observed reduction satisfies $\Fbar_k^0 - \Fbar_k^\s \ge 2\theta^2 \eta \Lambda_k \Delta_k^3$. Then, for all sufficiently large $k$, it follows from Theorem~\ref{thm:bounded-error} that
    \(
        f_k^0 - f_k^\s = (\Fbar_k^0 - \Fbar_k^\s) + (\Ebar_k^{\s} - \Ebar_k^{0}) 
        \ge 2\theta^2 \eta \Lambda_k \Delta_k^3 - 2\tilde\kappa_{ef} \Delta_k^3.
    \)
    By choosing $\kappa_{eg}\le (2\Lambda_{\min})^{-1}$ and $\kappa_{ef} \le \theta^2\eta\Lambda_{\min}/(256\sqrt{2})$ in Theorem~\ref{thm:bounded-error}, we get $\tilde\kappa_{ef} \le \theta^2 \eta \Lambda_{\min}/2$ and consequently obtain for sufficiently large $k$
    \begin{equation} \label{eq:true_fn_rd}    
        f_k^0 - f_k^\s \ge \theta^2 \eta \Lambda_k \Delta_k^3.
    \end{equation}
    Since the objective function is bounded below by $f^*$, this implies
    \begin{equation*}
        f^* < f_k^\s \le f_k^0 - \theta^2 \eta \Lambda_k \Delta_k^3 = f_k^0 - \frac{\theta^2 \eta}{64 \sqrt{\Lambda_k}} \|\BFGbar_k^0\|^{3/2} \le f^u - \frac{\theta^2 \eta}{64 \sqrt{\mu \|\BFGbar_k^0\|}} \|\BFGbar_k^0\|^{3/2},
    \end{equation*}
    where we substituted $\Delta_k = (\|\BFGbar_k^0\|/16\Lambda_k)^{1/2}$ and $\Lambda_k \le \mu \|\BFGbar_k^0\|$. This implies an upper bound on the estimated gradient as follows.
    \begin{equation} \label{eq:Gbar_ub}
        \|\BFGbar_k^0\| \le \left( \frac{64 \sqrt{\mu}}{\theta^2 \eta} (f^u - f^*) \right)^{2/3}.
    \end{equation}
   From~\eqref{eq:Gbar_ub} and Lemma~\ref{lem:bounded-error} with $\alpha = 0$, we observe that for sufficiently large $k$ and any $\kappa_{eg}>0$:  \begin{equation} \label{eq:true-g-bound-step}
        \|\nabla f_k^0\| \le \|\BFGbar_k^0\| + \|\nabla f_k^0 - \BFGbar_k^0\|
        \le \left( \frac{64 \sqrt{\mu}}{\theta^2 \eta} (f^u - f^*) \right)^{2/3} + \kappa_{eg} := \kappa_{u1}.
    \end{equation}
    Moreover, we know from the Lipschitz continuity of the gradient (Assumption~\ref{assum:smooth-f})
    \begin{equation*}
        \|\nabla f_k^\s - \nabla f_k^0\| \le \kappa_{Lg} \|\BFS_k\| \le \kappa_{Lg} \Delta_k = \frac{\kappa_{Lg}}{4\sqrt{\Lambda_k}} \sqrt{\|\BFGbar_k^0\|} \le \frac{\kappa_{Lg}}{4\sqrt{\mu}}.
    \end{equation*}
    Combining \eqref{eq:true-g-bound-step} with this variation, we obtain the bound for the candidate gradient as follows.
    \begin{equation} \label{eq:candidate_gradient_ub_final}
        \|\nabla f_k^\s \| \le \|\nabla f_k^0\| + \|\nabla f_k^\s - \nabla f_k^0\| 
        \le \kappa_{u1} + \frac{\kappa_{Lg}}{4\sqrt{\mu}}=: \kappa_{u2}.
    \end{equation}
    Finally, we examine the event $\{\|\nabla f_k^0\| > \kappa_{u1}\}$, implying that $k$ cannot be in $\mcD$. If $k \in \mcG$, we obtain from Theorem~\ref{thm:bounded-error} (with $\alpha = 2$) that
    \begin{align*} 
    \|\nabla f_k^\s\| - \tilde\kappa_{eg}\Delta_k^2
    \le \|\nabla f_k^\s\| - \|\nabla f_k^\s - \BFGbar_k^\s\|
    \le \|\BFGbar_k^\s\|
    \le \eta\|\BFGbar_k^{0}\| = 16 \eta \Lambda_k \Delta_k^2,
    \end{align*}
    Therefore, we have from Lemma~\ref{lem:bounded-error} (with $\alpha = 2$)
    \begin{align} \label{eq:true-gradient-contraction}   
    \|\nabla f_k^\s\| 
    &\le \left( 16\eta \Lambda_k + \tilde\kappa_{eg} \right) \Delta_k^2 
    \le \left(\eta + \frac{\tilde\kappa_{eg}}{16\Lambda_{\min}}\right) \|\BFGbar_k^0\| \nonumber \\
    &\le \underbrace{\left(\eta + \frac{\tilde\kappa_{eg}}{16\Lambda_{\min}}\right)\left(1+\frac{\kappa_{eg}}{\Lambda_{\min}}\right)}_{:=\bar{\eta}}\|\nabla f_k^0\|.
    \end{align}
    We choose $\kappa_{eg}$ sufficiently small such that
    \[
    \kappa_{eg} \le \min\left\{ \frac{1}{2\Lambda_{\min}}, \, \frac{(1-\eta)\Lambda_{\min}}{4} \right\}.
    \]
    Then $1-\kappa_{eg}\Lambda_{\min} \ge 1/2$, and using
    \(
    \tilde\kappa_{eg} \le \kappa_{eg}\left(16/(1-\kappa_{eg}\Lambda_{\min})\right)
    \)
    from Theorem~\ref{thm:bounded-error}, we obtain
    \[
    \frac{\tilde\kappa_{eg}}{16\Lambda_{\min}}
    \le
    \frac{\kappa_{eg}}{\Lambda_{\min}(1-\kappa_{eg}\Lambda_{\min})}
    \le
    \frac{2\kappa_{eg}}{\Lambda_{\min}}.
    \]
    Substituting this bound yields
    \[
    \bar\eta = \left(\eta + \frac{\tilde\kappa_{eg}}{16\Lambda_{\min}}\right)
    \left(1+\frac{\kappa_{eg}}{\Lambda_{\min}}\right)
    \le
    \left(\eta + \frac{2\kappa_{eg}}{\Lambda_{\min}}\right)
    \left(1+\frac{\kappa_{eg}}{\Lambda_{\min}}\right)
    \le
    \frac{1+\eta}{2}\cdot \frac{5-\eta}{4}.
    \]
    Since $\eta \in (1/4,1)$, it follows that $\bar{\eta} < 1$, which in turn guarantees a reduction in the gradient norm. If the iteration $k$ is unsuccessful, $\|\nabla f_{k+1}^0\| = \|\nabla f_k^0\|$. As a result, we have $\|\nabla f_k^0\| \le \kappa_{u2}$ for all sufficiently large $k$. Since $\lambda_k$ is an increasing sequence, for sufficiently large $k$, we have $\lambda_k^{1/3} \ge 1$, and hence the constant bound $\|\nabla f_k^0\| \le \kappa_{u2}$ implies
    \(
    \|\nabla f_k^0\| \le \kappa_{u2} \lambda_k^{1/3}.
    \)

    \textbf{Case 2: $\Lambda_k^{\max} = \lambda_k$.} 
    When $\Lambda_k \le \lambda_k$, the reduction in the objective function~\eqref{eq:true_fn_rd} yields
    \begin{equation*}
        f^* < f_k^\s \le f_k^0 - \theta^2 \eta \Lambda_k \Delta_k^3 = f_k^0 - \frac{\theta^2 \eta}{64 \sqrt{\Lambda_k}} \|\BFGbar_k^0\|^{3/2} \le f_k^0 - \frac{\theta^2 \eta}{64 \sqrt{\lambda_k}} \|\BFGbar_k^0\|^{3/2}.
    \end{equation*}
    Rearranging this inequality with $f_k^0 \le f^u$, we obtain an upper bound on the estimated gradient
    \begin{equation} \label{eq:Gbar_ub_lambda}
        \|\BFGbar_k^0\| \le \lambda_k^{1/3} \left( \frac{64}{\theta^2 \eta} (f^u - f^*) \right)^{2/3}.
    \end{equation}
    Following the same logic as in Case 1, we apply Lemma~\ref{lem:bounded-error} (with $\alpha=0$) to show that for any $\kappa_{eg} > 0$
    \begin{equation*}
        \|\nabla f_k^\s \| \le \|\nabla f_k^0\| + \|\nabla f_k^\s - \nabla f_k^0\| 
        \le \lambda_k^{1/3} \left( \frac{64}{\theta^2 \eta} (f^u - f^*) \right)^{2/3} + \kappa_{eg} + \frac{\kappa_{Lg}}{4\sqrt{\mu}}.
    \end{equation*}
    For $k \in \mathcal{G}$, the same contraction argument as in Case~1 applies. Therefore, the gradient norm decreases on iterations in $\mathcal{G}$ and remains unchanged on unsuccessful iterations, implying that $\|\nabla f_k^0\| = \mcO(\lambda_k^{1/3})$ for all sufficiently large $k$.
    
    Combining both cases, there exists a finite constant $\kappa_u$ such that $\|\nabla f_k^0\| \le \kappa_u \lambda_k^{1/3}$ holds almost surely for all sufficiently large $k$. Consequently, the event $\{\|\nabla f_k^0\| > \kappa_u \lambda_k^{1/3} \text{ i.o.}\}$ has probability zero, which completes the proof.
\end{proof}

\begin{remark}
    The assumption that the objective function $f(\cdot)$ is bounded above by $f^u$, as required in Lemma~\ref{lem:grad-ub-io}, can be relaxed without affecting the subsequent results. Extending the analysis to functions that are unbounded above is straightforward: the only modification is that the constant $\kappa_u$ becomes a finite positive random variable. Since this does not materially alter the upcoming proof arguments, we proceed without imposing the boundedness-from-above of $f$ assumption in the remainder of the analysis.
\end{remark}

As a direct consequence, the bound on $\Lambda_k$ in Lemma~\ref{lem:reg-ub} can be simplified to the following form.

\begin{corollary} \label{cor:reg-ub}
    Suppose that Assumptions~\ref{assum:smooth-f}--\ref{assum:lipschtizmodelHessian} hold.  
    Then for the iterates $\{\BFX_k\}$ generated by Algorithm~\ref{alg:ASUTRO}, 
    \begin{equation*}
        \mathbb{P} \left( \Lambda_k > \gamma_1 \lambda_k  \ \text{i.o.} \right) = 0.
    \end{equation*}    
\end{corollary}
\begin{proof}
    Let $\omega \in \Omega\backslash \Omega_0$ be fixed, where $\Omega_0$ is the largest set of measure 0. We omit the dependence on $\omega$ for simplicity.     
    There exists an index $K$ such that the following hold for all $k \ge K$:
    \begin{itemize}
        \item[(i)] By Lemma~\ref{lem:reg-ub}, $\Lambda_k \le \gamma_1 \max\{\Lambda_\mcD,\Lambda_\mcG,\mu\|\BFGbar_k^0\|, \lambda_k\}$;
        \item[(ii)] By Lemma~\ref{lem:grad-ub-io} and Theorem~\ref{thm:bounded-error}, 
        \(
        \|\BFGbar_k^0\| \le \|\nabla f_k^0\| + \|\BFGbar_k^0-\nabla f_k^0\| \le \kappa_{u} \lambda_k^{1/3} + \tilde\kappa_{eg}.
        \)
    \end{itemize}
    Hence, for all $k \ge K$,
    \(
        \Lambda_k
        \le \gamma_1 \max\{\Lambda_\mcD,\Lambda_\mcG,\mu\big(\kappa_u \lambda_k^{1/3}+\tilde\kappa_{eg}\big),\lambda_k\}.
    \)
    Since $\lambda_k = \mcO((\log k)^{1+\epsilon_\lambda})$, there exists an index $\widetilde K \ge K$ such that, for all $k \ge \widetilde K$, 
    \[
    \max\{\Lambda_\mcD,\Lambda_\mcG,\mu\big(\kappa_u \lambda_k^{1/3}+\tilde\kappa_{eg}\big),\lambda_k\} \le \lambda_k.
    \]
    Therefore, $\Lambda_k \le \gamma_1\lambda_k$ for all $k \ge \widetilde K$, which completes the proof.
\end{proof}

We next establish a key property of the true function values along the trajectory. In particular, iterations in $\mcD$ achieve sufficient function decrease, while iterations in $\mcG$, although defined through gradient-based conditions, do not increase the function value. As a result, the sequence $\{f(\BFX_k)\}$ is eventually non-increasing.

\begin{lemma} \label{lem:monotone_decreasing}
    Suppose Assumptions~\ref{assum:smooth-f}--\ref{assum:lipschtizmodelHessian} hold.  
    The sequence of iterates $\{\BFX_k\}$ generated by Algorithm~\ref{alg:ASUTRO} is almost surely eventually monotone. In particular, it satisfies the following:
    \begin{align}
    \mathbb{P}\!\left( (k \in \mcD) \bigcap (f_k^0 - f_k^{\s} < \theta^2 \eta \Lambda_k \Delta_k^{3}) \text{ i.o.} \right) &= 0, \label{eq:true_fn_rd_D}\\
    \mathbb{P}\!\left( (k \in \mcG) \bigcap (f_k^0 - f_k^{\s} < 0) \text{ i.o.} \right) &= 0. \label{eq:true_fn_rd_G}
    \end{align}
\end{lemma}

\begin{proof}
    Let $\omega \in \Omega$ denote a realization of Algorithm~\ref{alg:ASUTRO} with a nonzero probability; we omit $\omega$ for simplicity. We obtain from Taylor's expansion and Lemma~\ref{lem:model-reduction}
    \begin{align*}
        f_k^\s - f_k^0 &\le \BFS_k^\top \nabla f_k^0 + \frac{1}{2} \BFS_k^\top \nabla^2 f_k^0 \BFS_k + \frac{\kappa_{L\sfH}}{6} \|\BFS_k\|^3  \\
        &= \BFS_k^\top (\nabla f_k^0 - \BFGbar_k^0) + \frac{1}{2} \BFS_k^\top (\nabla^2 f_k^0 - \sfH_k) \BFS_k + \frac{\kappa_{L\sfH}}{6} \|\BFS_k\|^3 + \left( \BFS_k^\top \BFGbar_k^0 + \frac{1}{2} \BFS_k^\top \sfH_k \BFS_k \right) \\
        &\le \BFS_k^\top \BFEbar_k^{g,0} + \frac{1}{2} \BFS_k^\top (\nabla^2 f_k^0 - \sfH_k) \BFS_k + \frac{\kappa_{L\sfH}}{6} \|\BFS_k\|^3 - \frac{1}{2} \sqrt{\Lambda_k \|\BFGbar_k^0\|} \|\BFS_k\|^2.
    \end{align*}
    Applying Theorem~\ref{thm:bounded-error} and Corollary~\ref{cor:hessian-error} and substituting $\sqrt{\Lambda_k \|\BFGbar_k^0\|} = 4\Lambda_k \Delta_k$, we have for all sufficiently large $k$,
    \begin{equation} \label{eq:true-reduction}
        f_k^\s-f_k^0\le \|\BFS_k\|^2\Delta_k\left(\tilde\kappa_{eg}\frac{\Delta_k^{\alpha-1}}{\|\BFS_k\|}+\frac{\kappa_{e\sfH}}{2}+\frac{\kappa_{L\sfH}}{6}-2\Lambda_k\right).
    \end{equation}
    To complete the proof, we consider three cases depending on whether $k \in \mcD$ or $k \in \mcG$, and, in the latter case, whether $\ell_k$ is positive or zero.
    
    \textbf{Case 1: $k \in \mcD$.} 
    Following the same argument as in the proof of Lemma~\ref{lem:grad-ub-io}, which combines the success condition $\rho_k \ge \eta$, the model reduction bound, and Theorem~\ref{thm:bounded-error}, for sufficiently large $k$, we obtain (see \eqref{eq:true_fn_rd})
    \(
    f_k^0 - f_k^\s \ge \theta^2 \eta \Lambda_k \Delta_k^3.
    \)
    This establishes~\eqref{eq:true_fn_rd_D}.
    
    \textbf{Case 2: $k \in \mcG$ and $\ell_k > 0$.} 
    From \eqref{eq:true-reduction} with $\alpha = 2$ and $\|\BFS_k\| = \Delta_k$, we have 
    \begin{equation*}
        f_k^\s-f_k^0\le \Delta_k^3\left(\tilde\kappa_{eg}+\frac{\kappa_{e\sfH}}{2}+\frac{\kappa_{L\sfH}}{6}-2\Lambda_k\right).
    \end{equation*}
    As a result, when $\Lambda_k \ge \frac{\tilde\kappa_{eg}}{2} + \frac{\kappa_{e\sfH}}{4}+\frac{\kappa_{L\sfH}}{12}$, we have $f_k^0 - f_k^\s \ge 0$. 
    
    \textbf{Case 3: $k \in \mcG$ and $\ell_k = 0$.}
    We first establish a lower bound for the step-size. From the optimality condition \eqref{eq:opt-3}, Assumption~\ref{assum:lipschtizmodelHessian}, and $\|\BFGbar_k^0\| = 16 \Lambda_k \Delta_k^2$, we have:
    \begin{equation} \label{eq:Sk_lowerbound}
        \| \BFS_k \| \ge \frac{\|\BFGbar_k^0\|}{\|\sfH_k\| + \sqrt{\Lambda_k \|\BFGbar_k^0\|} + \ell_k} \ge \frac{16 \Lambda_k \Delta_k^2}{\kappa_{\sfH} + 4 \Lambda_k \Delta_k}.
    \end{equation}
    From~\eqref{eq:Sk_lowerbound} and the fact that $\Lambda_k \ge \mu \|\BFGbar_k^0\|$ in $\mcG$, we have
    \begin{equation*}
        \frac{\Delta_k^2}{\|\BFS_k\|} \le \frac{\kappa_{\sfH} + \sqrt{\Lambda_k \|\BFGbar_k^0\|}}{16 \Lambda_k} \le \frac{\kappa_{\sfH}}{16 \Lambda_{\min}} + \frac{1}{16 \mu}.
    \end{equation*}
    Substituting this relationship back into the function change bound, we obtain:
    \begin{equation*}
        f_k^\s - f_k^0 \le \left( 
        \underbrace{\frac{\tilde\kappa_{eg}}{16}
        \left( \frac{\kappa_{\sfH}}{\Lambda_{\min}} + \frac{1}{\mu} \right) + \frac{\kappa_{e\sfH}}{2} + \frac{\kappa_{L\sfH}}{6}}_{:=\kappa_{e}} - 2\Lambda_k \right) \|\BFS_k\|^2 \Delta_k.
    \end{equation*}
    As a result, when $\Lambda_k \ge \kappa_{e}/2$, we have $f_k^0 - f_k^\s \ge 0$. 
    For $k \in \mcG$, since $\Lambda_k \ge \lambda_k$ and $\lambda_k \to \infty$, 
    it follows that $\Lambda_k \ge \kappa_e/2$ for all sufficiently large $k$.
    Combining Cases~2 and~3, we establish~\eqref{eq:true_fn_rd_G}.
\end{proof}

A key step in bounding the cardinality of $\mcG$ is to characterize the behavior of the true gradient along these iterations. The following lemma establishes a contraction property of the true gradient norm that, together with Lemma~\ref{lem:grad-ub-io}, enables a bound on the number of iterations in $\mcG$.

\begin{lemma} \label{lem:true-gradient-contraction} 
    Suppose Assumptions~\ref{assum:smooth-f}--\ref{assum:lipschtizmodelHessian} hold. For iterates $\{\BFX_k\}$  generated by Algorithm~\ref{alg:ASUTRO}, there exists a constant $\bar{\eta} \in (0, 1)$ such that
    \begin{equation*}
    \mathbb{P} \left( (k \in \mcG) \bigcap (\|\nabla f_k^\s\| > \bar{\eta} \|\nabla f_k^0\|) \text{ i.o.} \right) = 0.
    \end{equation*}
\end{lemma}

\begin{proof}
    Let $\omega \in \Omega\backslash \Omega_0$ be fixed, where $\Omega_0$ is the largest set of measure 0; we omit $\omega$ for simplicity.
    For $k \in \mcG$, the algorithm ensures that 
    \(
    \|\BFGbar_k^\s\| \le \eta \|\BFGbar_k^0\|.
    \)
    Following the same argument used in the proof of Lemma~\ref{lem:grad-ub-io} (see~\eqref{eq:true-gradient-contraction}), we obtain for all sufficiently large $k$ that
    \(
    \|\nabla f_k^\s\|
    \le
     \bar{\eta}
    \|\nabla f_k^0\|,
    \)
    which completes the proof.
\end{proof}

\begin{remark} \label{rmk:lambda-growth}
The following growth property of the sequence $\{\lambda_k\}$ will be used repeatedly in the analysis. 
From the definition of $\{\lambda_k\}$ in Algorithm~\ref{alg:ASUTRO}, we choose 
\(
\lambda_k = \mcO\big((\log k)^{1+\epsilon_\lambda}\big).
\)
In particular, there exists a constant $\kappa_\lambda > 0$ such that \( \lambda_k \le \kappa_\lambda (\log k)^{1+\epsilon_\lambda} \) for all $k \in \mbN$.
\end{remark}

We now turn to bounding the number of iterations in $\mcD$ and $\mcG$ up to $T_{\epsilon}$, using the results developed above.

\begin{lemma} \label{lem:D-eps-count}
    Suppose Assumptions~\ref{assum:smooth-f}--\ref{assum:lipschtizmodelHessian} hold. 
    There exists a positive random variable $M_{\mcD}$ such that for any deterministic sequence $\{\epsilon_{\ell}\}$ with $\epsilon_{\ell} \downarrow 0$,
    \begin{equation*}
        \mathbb{P}\!\left( |\mcD_{\epsilon_{\ell}}| > M_{\mcD}\epsilon_{\ell}^{-3/2}\left(\log T_{\epsilon_{\ell}} \right)^{\frac{1+\epsilon_\lambda}{2}} \text{ i.o.} \right) = 0,
    \end{equation*}
    where $\mcD_{\epsilon_{\ell}} := \mcD \cap \{k < T_{\epsilon_{\ell}}\}$ and $T_{\epsilon_{\ell}}$ is the first iteration index such that $\|\nabla f_k^0\| \le \epsilon_{\ell}$.
\end{lemma}

\begin{proof}
    Let $\omega \in \Omega$ denote a realization of Algorithm~\ref{alg:ASUTRO} with nonzero probability; we omit the dependence on $\omega$ for simplicity. Consider $k \in \mcD$. From Lemma~\ref{lem:monotone_decreasing} and Corollary~\ref{cor:reg-ub}, there exists an index $K$ such that for all $k \ge K$, we have
    \begin{align*}
        f_k^0 - f_k^\s &\ge \theta^2 \eta \Lambda_k \Delta_k^3 \ge \frac{\theta^2 \eta}{64 \sqrt{\Lambda_k}} \|\BFGbar_k^0\|^{3/2} 
        \ge \underbrace{\frac{\theta^2 \eta}{64 \sqrt{\gamma_1}} \left( 1 - \frac{\kappa_{eg}}{\Lambda_{\min}^{3/2}} \right)}_{:=\kappa_{\mcD}} \frac{\|\nabla f_k^0\|^{3/2}}{\lambda_k}.
    \end{align*}
    Fix $\epsilon_0 > 0$ sufficiently small such that for all $\epsilon < \epsilon_0$, $T_\epsilon > K.$
    For $k < T_\epsilon$, we have $\|\nabla f_k^0\| > \epsilon$. Furthermore, since $\lambda_k$ is not decreasing, $\lambda_k \le \lambda_{T_\epsilon}$ for all $k < T_\epsilon$. Consequently, for $k \in \mcD_\epsilon \cap \{k \ge K\}$, 
    \begin{equation} \label{eq:epsilon-descent}
        f_k^0 - f_k^\s \ge \kappa_{\mcD} \frac{\epsilon^{3/2}}{\sqrt{\lambda_{T_\epsilon}}}.
    \end{equation}
    Summing \eqref{eq:epsilon-descent} over $k \in \mathcal{D}_\epsilon \cap \{k \ge K\}$ and applying the telescoping sum of the function values, we have
    \begin{equation} \label{eq:mcDsumK}
        \frac{\kappa_{\mcD} \epsilon^{3/2}}{\sqrt{\lambda_{T_\epsilon}}} |\mathcal{D}_\epsilon \cap \{k \ge K\}| \le \sum_{k \in \mathcal{D}_\epsilon \cap \{k \ge K\}} (f_k^0 - f_k^\s) \le f_K^0 - f^*.
    \end{equation}
    Furthermore, since the single-step reduction $f_k^0 - f_k^\s$ cannot exceed the total available descent $f_K^0 - f^*$, the term $\epsilon^{3/2}/\sqrt{\lambda_{T_\epsilon}}$ is inherently bounded by $(f_K^0 - f^*) / \kappa_{\mcD}$. 
    Rearranging~\eqref{eq:mcDsumK} and accounting for the first $K$ iterations, we obtain
    \begin{equation} \label{eq:mcDsum}
        |\mathcal{D}_\epsilon| \frac{\epsilon^{3/2}}{\sqrt{\lambda_{T_\epsilon}}} \le \frac{f_K^0 - f^*}{\kappa_{\mcD}} + K \frac{\epsilon^{3/2}}{\sqrt{\lambda_{T_\epsilon}}} \le (1+K) \frac{f_K^0 - f^*}{\kappa_{\mcD}}.
    \end{equation}
    By Remark~\ref{rmk:lambda-growth}, we obtain 
    \(
    \sqrt{\lambda_{T_\epsilon}} \le \sqrt{\kappa_{\lambda}} (\log T_\epsilon)^{\delta},
    \)
    where $\delta = (1+\epsilon_\lambda)/2$. Substituting this into \eqref{eq:mcDsum} yields
    \[
    |\mcD_\epsilon|
    \le
    \underbrace{\left((1+K) \frac{f_K^0 - f^*}{\kappa_{\mcD}}\right)}_{:=M_{\mcD}}\epsilon^{-3/2}
    (\log T_\epsilon)^{\delta}.
    \]
    This completes the proof.
\end{proof}

\begin{lemma} \label{lem:G-eps-count}
    Suppose Assumptions~\ref{assum:smooth-f}--\ref{assum:lipschtizmodelHessian} hold. 
    There exists a positive random variable $M_{\mcG}$ such that for any deterministic sequence $\{\epsilon_{\ell}\}$ with $\epsilon_{\ell} \downarrow 0$,
    \begin{equation*}
    \mbP\left(|\mcG_{\epsilon_{\ell}}| > M_{\mcG} \left( \log(1/\epsilon_{\ell}) + \log \log T_{\epsilon_{\ell}} \right) |\mcD_{\epsilon_{\ell}}| \text{ i.o.}  \right)=0,
    \end{equation*}
    where $\mcG_{\epsilon_{\ell}} := \mcG \cap \{k < T_{\epsilon_{\ell}}\}$ and $T_{\epsilon_{\ell}} := \min \{k : \|\nabla f_k^0\| \le \epsilon_{\ell}\}$.
\end{lemma}

\begin{proof}
    Let $\omega \in \Omega$ denote a realization of Algorithm~\ref{alg:ASUTRO} with nonzero probability; we omit the dependence on $\omega$ for simplicity.  From Lemma~\ref{lem:grad-ub-io} and~\ref{lem:true-gradient-contraction}, there exists a finite index $K$ such that for all $k \ge K$, the following conditions are satisfied
    \begin{itemize}
        \item[(i)] For $k \in \mcG$, $\|\nabla f_{k+1}^{0}\| \le \bar{\eta} \|\nabla f_{k}^{0}\|$ with $\bar{\eta} \in (0, 1)$;
        \item[(ii)] For all $k$, $\|\nabla f_k^0\| \le \kappa_u \lambda_k^{1/3}$.
    \end{itemize}
    Fix $\epsilon_0 > 0$ sufficiently small such that for all $\epsilon < \epsilon_0$, $T_\epsilon > K.$ Let $\mcD_\epsilon = \{k_1, k_2, \dots, k_m\}$ be the indices in $\mcD$ prior to $T_\epsilon$, where $k_1 < k_2 < \dots < k_m < T_\epsilon$ and $m = |\mcD_\epsilon|$. We partition the successful iterations $\mcG_\epsilon$ into intervals defined by $\mcD_\epsilon$. Let $\mcG^{(i)} = \mcG \cap \{k_i, k_i+1, \dots, k_{i+1}-1\}$ for $i=1, \dots, m-1$, and $\mcG^{(m)} = \mcG \cap \{k_m, \dots, T_\epsilon-1\}$.
    For any $k \in \mcG^{(i)}$, the gradient norm contracts by $\bar{\eta}$. Thus, for the total number of successful steps in the $i$-th interval, denoted by $q_i = |\mcG^{(i)}|$, we have
    \begin{equation*}
        \|\nabla f_{k_{i+1}}^0\| \le (\bar{\eta})^{q_i} \|\nabla f_{k_i}^0\| \implies q_i \le \frac{\log(\|\nabla f_{k_i}^0\| / \|\nabla f_{k_{i+1}}^0\|)}{\log(1/\bar{\eta})}.
    \end{equation*}
    By the definition of $T_\epsilon$, we have $\|\nabla f_k^0\| > \epsilon$ for all $k < T_\epsilon$. 
    Combining this with condition (ii), for $k_i \ge K$, we obtain
    \[
    q_i 
    \le 
    \frac{\log(\kappa_u \lambda_{k_i}^{1/3} / \epsilon)}{\log(1/\bar{\eta})}.
    \]
    Since $k_i < T_\epsilon$ and $\lambda_k$ is non-decreasing, it follows that $\lambda_{k_i} \le \lambda_{T_\epsilon}$. In addition, we know from Remark~\ref{rmk:lambda-growth} \( \log \lambda_{T_\epsilon} \le \log \kappa_\lambda + (1+\epsilon_\lambda)\log\log T_\epsilon. \) 
    Hence, we obtain
    \[
    q_i
    \le
    \frac{
    \log \kappa_u + \tfrac{1}{3}\log \kappa_\lambda + \tfrac{(1+\epsilon_\lambda)}{3}\log\log T_\epsilon + \log(1/\epsilon)
    }{
    \log(1/\bar{\eta})
    }.
    \]
    Summing over all $i$ such that $k_i \ge K$, we obtain
    \begin{align*}
    |\mcG_\epsilon \cap \{k \ge K\}|
    \le
    |\mcD_\epsilon|
    \left(
    \frac{
    \log \kappa_u + \tfrac{1}{3} \log \kappa_\lambda + \tfrac{(1+\epsilon_\lambda)}{3}\log\log T_\epsilon + \log(1/\epsilon)
    }{
    \log(1/\bar{\eta})
    }
    \right).
    \end{align*}
    As a result, we can find a finite random variable $M_{\mcG}(\omega)$ that absorbs $K$ and the constant coefficients. Thus,
    \(
        |\mcG_\epsilon| \le M_{\mcG} \left( \log(1/\epsilon) + \log \log T_\epsilon \right) |\mcD_\epsilon|,
    \)
    which completes the proof.
\end{proof}

We now introduce the set of remaining iterations that do not satisfy the conditions defining $\mcD$ or $\mcG$. Specifically, define
\begin{align*}
\mcU := \{k \in \mathbb{N} : k \notin \mcD \cup \mcG \}.
\end{align*}
Thus, $\mcU$ consists of iterations that fail to achieve either a sufficiently large and successful step (in terms of both function reduction and step size) or a sufficient reduction in the gradient magnitude. We now analyze the number of iterations in $\mcU$ up to $T_\epsilon$ in terms of the occurrences of $\mcD$ and $\mcG$.

\begin{lemma} \label{lem:U-eps-count}
    Suppose Assumptions~\ref{assum:smooth-f}--\ref{assum:lipschtizmodelHessian} hold. 
    There exists a positive random variable $M_{\mcU}$ such that for any deterministic sequence $\{\epsilon_{\ell}\}$ with $\epsilon_{\ell} \downarrow 0$,
    \begin{equation*}
    \mbP\left(|\mcU_{\epsilon_{\ell}}| > M_{\mcU} \left(\log \log T_{\epsilon_{\ell}}\right) \left(|\mcD_{\epsilon_{\ell}}| + |\mcG_{\epsilon_{\ell}}| \right) \text{ i.o.}  \right)=0,
    \end{equation*}
    where $\mcU_{\epsilon_{\ell}} := \mcU \cap \{k < T_{\epsilon_{\ell}}\}$ and $T_{\epsilon_{\ell}} := \min \{k : \|\nabla f_k^0\| \le \epsilon_{\ell}\}$.
\end{lemma}
\begin{proof}
    Let $\omega \in \Omega$ denote a realization of Algorithm~\ref{alg:ASUTRO} with a nonzero probability; we omit $\omega$ for simplicity. There exists an index $K$ such that for all $k \ge K$:
    \begin{itemize}
        \item[(i)] By Corollary~\ref{cor:reg-ub}, $\Lambda_k \le \gamma_1 \lambda_k$;
        \item[(ii)] By Lemma~\ref{lem:grad-ub-io} and Theorem~\ref{thm:bounded-error}, 
        \(
        \|\BFGbar_k^0\| \le \|\nabla f_k^0\| + \|\BFGbar_k^0-\nabla f_k^0\| \le \kappa_{u} \lambda_k^{1/3} + \kappa_{eg}.
        \)
    \end{itemize}
    Note that the algorithm increases $\Lambda_k$ by a factor $\gamma_1 > 1$ only when $k \in \mcU$, and decreases it only when $k \in \mcG \cup \mcD$. Let $q_{\mcU}$ be the number of consecutive unsuccessful iterations between any two successful iterations. Starting from $\Lambda_{\min}$, the value of $\Lambda_k$ after $q_{\mcU}$ steps is at least $\Lambda_{\min} \gamma_1^{q_{\mcU}-1}$. Since $\Lambda_k$ cannot exceed $\gamma_1 \lambda_k$, we have
    \begin{equation*}
        \Lambda_{\min} \gamma_1^{q_{\mcU}-1} \le \gamma_1 \lambda_k \implies q_{\mcU} \le \frac{\log(\gamma_1^2 \lambda_k / \Lambda_{\min})}{\log \gamma_1} + 1.
    \end{equation*}
    Fix $\epsilon_0 > 0$ sufficiently small such that for all $\epsilon < \epsilon_0$, $T_\epsilon > K.$
    By Remark~\ref{rmk:lambda-growth}, for all $k \in [K,T_\epsilon)$, the number of consecutive unsuccessful steps is bounded by
    \begin{align*}
        q_{\mcU} 
        \le \frac{\log(\gamma_1 \lambda_k / \Lambda_{\min})}{\log \gamma_1} + 1 
        \le \frac{\log \kappa_{\lambda} + (1+\epsilon_\lambda) \log \log T_\epsilon +  \log(\gamma_1 / \Lambda_{\min})}{\log \gamma_1} + 1,
    \end{align*}
    The total number of unsuccessful iterations $|\mcU_\epsilon \cap \{k \ge K\}|$ is the sum of these consecutive sequences. Since each such sequence must be preceded by an iteration in $\mcD \cup \mcG$ (or the initial index $K$), the number of sequences is at most $|\mcD_\epsilon| + |\mcG_\epsilon| + 1$. Thus,
    \begin{align*}
        |\mcU_\epsilon \cap \{k \ge K_\lambda\}| 
        &\le (|\mcD_\epsilon| + |\mcG_\epsilon| + 1) 
        \left( \frac{\log \kappa_{\lambda} + (1+\epsilon_\lambda) \log \log T_\epsilon + \log(\gamma_1 / \Lambda_{\min})}{\log \gamma_1} + 1 \right) \\
        &\le (|\mcD_\epsilon| + |\mcG_\epsilon| + 1) \log \log T_\epsilon \left( \frac{1+\epsilon_\lambda}{\log \gamma_1} + \frac{\log( \kappa_{\lambda} \gamma_1^2 / \Lambda_{\min})}{(\log \gamma_1) \log \log T_\epsilon} \right).
    \end{align*}
    As $\epsilon \downarrow 0$, we have $T_\epsilon \to \infty$, causing the second term in the parenthesis to vanish.  By defining a sufficiently large random variable $M_{\mcU}$ to account for the initial iterations $k < K$ and the constants above, we obtain
    \begin{equation*}
        |\mcU_\epsilon| \le M_{\mcU} (\log \log T_\epsilon) (|\mcD_\epsilon| + |\mcG_\epsilon|),
    \end{equation*}
    completing the proof.
\end{proof}

We now combine the bounds on $\mcD_{\epsilon}$, $\mcG_{\epsilon}$, and $\mcU_{\epsilon}$ to bound $T_{\epsilon}$ via
\(
T_{\epsilon} = |\mcD_{\epsilon}| + |\mcG_{\epsilon}| + |\mcU_{\epsilon}|
\)
and
$W_\epsilon$ via 
\(
\sum_{k=0}^{T_{\epsilon_{\ell}}-1} (N_k + N_k^\s).
\)
This yields the following sample and iteration complexity results.

\begin{theorem}[Iteration Complexity] \label{thm:T-eps-rate}
    Suppose Assumptions~\ref{assum:smooth-f}--\ref{assum:lipschtizmodelHessian} hold. 
    There exists a positive random variable $M_T > 0$ such that for any deterministic sequence $\{\epsilon_{\ell}\}$ with $\epsilon_{\ell} \downarrow 0$:
    \begin{equation*}
        \mathbb{P}\!\left( T_{\epsilon_{\ell}}  > M_T \epsilon_{\ell}^{-3/2} (\log(1/\epsilon_{\ell}))^{1+\delta} \log\log(1/\epsilon_{\ell})\text{ i.o.} \right) = 0,
    \end{equation*}
    where $\delta = (1+\epsilon_\lambda)/2 > 0$ and and $T_{\epsilon_{\ell}} := \min \{k : \|\nabla f_k^0\| \le \epsilon_{\ell}\}$.
\end{theorem}

\begin{proof}
    Let $\omega \in \Omega$ denote a realization of Algorithm~\ref{alg:ASUTRO} with a nonzero probability; we omit $\omega$ for simplicity. Fix $\epsilon_0 \in (0,1)$ sufficiently small such that for all $\epsilon < \epsilon_0$: 
    \begin{itemize}
        \item[(a)] $|\mcD_\epsilon| \le M_{\mcD} \epsilon^{-3/2}\left(\log T_{\epsilon} \right)^{\delta}$ by Lemma~\ref{lem:D-eps-count};
        \item[(b)] $|\mcG_{\epsilon}| \le M_{\mcG} \left( \log(1/\epsilon) + \log \log T_{\epsilon} \right) |\mcD_{\epsilon}|$ by Lemma~\ref{lem:G-eps-count};
        \item[(c)] $|\mcU_{\epsilon}| \le M_{\mcU} \left(\log \log T_{\epsilon}\right) \left(|\mcD_{\epsilon}| + |\mcG_{\epsilon}|\right)$ by Lemma~\ref{lem:U-eps-count};
        \item[(d)] $\log(1/\epsilon) \ge 1$ and $\log\log T_\epsilon \ge 1$;
        \item[(e)] $\left(\log T_{\epsilon} \right)^{\delta} (\log\log T_\epsilon)^2 < \sqrt{T_\epsilon}$.
    \end{itemize}
    Conditions (d) and (e) are naturally satisfied for sufficiently small $\epsilon$ because $T_\epsilon \to \infty$ as $\epsilon \to 0$, and any polylogarithmic function grows slower than $\sqrt{T_\epsilon}$.

     We first note that the total number of iterations until $\|\nabla f_k^0\| \le \epsilon$ is given by$$T_{\epsilon} \le |\mcD_\epsilon| + |\mcG_\epsilon| + |\mcU_\epsilon|.$$Substituting the bound for $|\mcU_\epsilon|$ from condition (d), we have:
     \begin{align*}
     T_{\epsilon} 
     \le |\mcD_\epsilon| + |\mcG_\epsilon| + M_{\mcU} (\log \log T_\epsilon) (|\mcD_\epsilon| + |\mcG_\epsilon|) 
     = (1 + M_{\mcU} \log \log T_\epsilon) (|\mcD_\epsilon| + |\mcG_\epsilon|).
     \end{align*}
     Next, we substitute the bound for $|\mcG_\epsilon|$ from condition (c):
     \begin{align*}
     T_{\epsilon} &\le (1 + M_{\mcU} \log \log T_\epsilon) \left( |\mcD_\epsilon| + M_{\mcG} (\log(1/\epsilon) + \log \log T_\epsilon) |\mcD_\epsilon| \right) \\
     &= |\mcD_\epsilon| (1 + M_{\mcU} \log \log T_\epsilon) (1 + M_{\mcG} (\log(1/\epsilon) + \log \log T_\epsilon)).
     \end{align*}
     Applying the bound for $|\mcD_\epsilon|$ from condition (b), specifically $|\mcD_\epsilon| \le M_{\mcD} \epsilon^{-3/2} \lambda_{T_\epsilon}^{1/2}$, we obtain:
     \begin{equation} \label{eq:T-eps-full-bound}
     T_{\epsilon} 
     \le M_{\mcD} \epsilon^{-3/2} \left(\log T_{\epsilon} \right)^{\delta} (1 + M_{\mcU} \log \log T_\epsilon) (1 + M_{\mcG} (\log(1/\epsilon) + \log \log T_\epsilon)).
     \end{equation}
     From \eqref{eq:T-eps-full-bound}, we expand the product and collect terms. 
    Absorbing the finite random constants $M_\mcG, M_{\mcU},$ and $M_\mcD$ into a generic finite random constant $M$ and expanding the product, we obtain
    \begin{align*}
    T_{\epsilon} \epsilon^{3/2}
    &\le M (\log T_{\epsilon})^{\delta} 
    \Big[
    1 
    + \log(1/\epsilon) 
    + \log \log T_\epsilon 
    + \log(1/\epsilon)\log \log T_\epsilon 
    + (\log \log T_\epsilon)^2
    \Big].
    \end{align*}

    Since $\log(1/\epsilon) \ge 1$ and $\log\log T_\epsilon \ge 1$ for sufficiently small $\epsilon$, the lower-order terms can be dominated by the higher-order ones. Hence, after enlarging $M$ if necessary, we obtain
    \begin{equation} \label{eq:tepsilon_reduced}
        T_{\epsilon} \epsilon^{3/2}
        \le M (\log T_{\epsilon})^{\delta}
        \left[
        \log(1/\epsilon)\log \log T_\epsilon
        + (\log \log T_\epsilon)^2
        \right].
    \end{equation}
    By condition (e), both
    \(
    (\log T_\epsilon)^{\delta}\log\log T_\epsilon < \sqrt{T_\epsilon}
    \)
    and
    \(
    (\log T_\epsilon)^{\delta}(\log\log T_\epsilon)^2 < \sqrt{T_\epsilon}
    \)
    hold for all sufficiently small $\epsilon$. Substituting these bounds into~\eqref{eq:tepsilon_reduced} yields
    \begin{align*}
        T_\epsilon \epsilon^{3/2}
        \le M \left( \log(1/\epsilon)\sqrt{T_\epsilon} + \sqrt{T_\epsilon} \right) 
        \le \widetilde M \log(1/\epsilon)\sqrt{T_\epsilon},
    \end{align*}
    for some finite random variable $\widetilde M>0$. Therefore,
    \begin{equation} \label{eq:coarse_derivation}
        T_\epsilon \le \widetilde M^2 \epsilon^{-3}(\log(1/\epsilon))^2.
    \end{equation}
    Taking logarithms in~\eqref{eq:coarse_derivation}, we obtain
    \[
    \log T_\epsilon
    \le 2\log \widetilde M + 3\log(1/\epsilon) + 2\log\log(1/\epsilon)
    \le 4\log(1/\epsilon),
    \]
    for all sufficiently small $\epsilon$. Consequently,
    \(
    \log\log T_\epsilon \le \log\bigl(4\log(1/\epsilon)\bigr)
    = \mathcal{O}(\log\log(1/\epsilon)),
    \)
    and
    \(
    (\log T_\epsilon)^\delta = \mathcal{O}\bigl((\log(1/\epsilon))^\delta\bigr).
    \)
    Substituting these back into \eqref{eq:tepsilon_reduced}, we obtain
    \begin{align*}
    T_\epsilon \epsilon^{3/2}
    &\le M (\log T_\epsilon)^\delta
    \left[
    \log(1/\epsilon)\log\log T_\epsilon
    + (\log\log T_\epsilon)^2
    \right] \\
    &= \mathcal{O}\!\left(
    (\log(1/\epsilon))^{1+\delta}\log\log(1/\epsilon)
    \right)
    +
    \mathcal{O}\!\left(
    (\log(1/\epsilon))^\delta (\log\log(1/\epsilon))^2
    \right).
    \end{align*}
    Since $\log(1/\epsilon)$ grows faster than $\log\log(1/\epsilon)$ as $\epsilon \downarrow 0$, the first term dominates the second for sufficiently small $\epsilon$. Therefore, there exists an almost surely finite random variable $M_T>0$ such that
    \[
    T_\epsilon \epsilon^{3/2}
    \le M_T (\log(1/\epsilon))^{1+\delta}\log\log(1/\epsilon),
    \]
    which proves the claim.
\end{proof}

\begin{theorem}[Sample Complexity] \label{thm:W-eps-rate}
    Suppose Assumptions~\ref{assum:smooth-f}--\ref{assum:lipschtizmodelHessian} hold. 
    There exists a positive random variable $M_W > 0$ such that for any deterministic sequence $\{\epsilon_{\ell}\}$ with $\epsilon_{\ell} \downarrow 0$:
    \begin{equation*}
        \mathbb{P}\!\left( W_{\epsilon_{\ell}}  > M_W \epsilon_{\ell}^{-9/2} (\log(1/\epsilon_{\ell}))^{\tilde \delta} (\log\log(1/\epsilon_{\ell}))^3 \text{ i.o.} \right) = 0,
    \end{equation*}
    where $\tilde \delta := \frac{17+11 \epsilon_\lambda}{2}$ and $W_{\epsilon_{\ell}} := \sum_{k=0}^{T_{\epsilon_{\ell}}-1} (N_k + N_k^\s)$.
\end{theorem}

\begin{proof}
    Let $\omega \in \Omega$ denote a realization of Algorithm~\ref{alg:ASUTRO} with a nonzero probability; we omit $\omega$ for simplicity. 
    By Theorem~2.8 of~\cite{shashaani2018astro}, for any $i\in\{0,\mathrm{s}\}$ we have
    $\widehat\sigma_{F}(\BFX_k^i,N_k^i)\to \sigma_F(\BFX_k^i)$ almost surely as $k\to\infty$; similarly for the gradient estimator, $\widehat\sigma_{\BFG}(\BFX_k^i,N_k^i)\to\sigma_{\BFG}(\BFX_k^i)$ almost surely as $k\to\infty$. 
    Hence, there exists an index $K_{fg}$ such that, for all $k\ge K_{fg}$ and $i\in\{0,\mathrm{s}\}$,
    \(
    \widehat\sigma_F^2(\BFX_k^i,N_k^i)\le 2\,\sigma_F^2
    \text{ and }
    \widehat\sigma_{\BFG}^2(\BFX_k^i,N_k^i)\le 2d\,\sigma_g^2.
    \)
    Set $\bar\sigma^2:=\max\{\sigma_0^2, 2\sigma_F^2, 2d \sigma_g^2\}$. Then, from \eqref{eq:as}, we obtain
    \[
    \sqrt{N_k} 
    \le \frac{\bar\sigma\sqrt{\lambda_k}}{\kappa_a (\Deltatilde_k(N_k))^3}
    \le \frac{\bar\sigma\sqrt{\lambda_k}\Lambda_k^{3/2}}{\kappa_a \varepsilon_k^{3/2}},
    \]
    where the last inequality comes from $\Deltatilde_k(n) \ge \sqrt{\varepsilon_k/\Lambda_k}$ for any $n \in \mbN$.
    By the argument in the proof of Lemma~\ref{lem:U-eps-count}, there exists a (random) index $K_\lambda$ such that for all $k \ge K_\lambda$,
    \(
    \Lambda_k \le \gamma_1 \lambda_k.
    \)
    Set $\delta = (1+\epsilon_{\lambda})/2$ and fix $\epsilon_0 > 0$ sufficiently small such that for all $\epsilon < \epsilon_0$: 
    \begin{itemize}
        \item[(a)] $T_\epsilon > K_w := \max\{K_{fg}, K_{\lambda}\}$;
        \item[(b)] $T_\epsilon \le M_T \epsilon^{-3/2} (\log(1/\epsilon))^{1+\delta}\log\log(1/\epsilon)$ by Theorem~\ref{thm:T-eps-rate}.
    \end{itemize}
    For $k \in (K_w, T_\epsilon]$, using $\varepsilon_k = c_* k^{-2/3}$, the number of samples at each iteration is bounded by:
    \begin{equation*}
        N_k^i \le \left\lceil \frac{\bar{\sigma}^2 \gamma_1^3 \lambda_k^4}{\kappa_a^2 \varepsilon_k^3} \right\rceil \le \left\lceil \kappa_n \lambda_{T_\epsilon}^4 T_\epsilon^2 \right\rceil = \left\lceil \kappa_n (\log T_\epsilon)^{4+4\epsilon_\lambda} T_\epsilon^2 \right\rceil,
    \end{equation*}
    where $\kappa_n := \bar{\sigma}^2 \gamma_1^3 c_*^{-3} \kappa_{a}^{-2}$. The total work $W_\epsilon$ up to the hitting time $T_\epsilon$ is then
    \begin{equation} \label{eq:wepsilon-sum}
    \begin{split}    
        W_\epsilon = \sum_{k=0}^{T_\epsilon-1} (N_k + N_k^\s) 
        &\le Q_w + 2 \sum_{k=K_w}^{T_\epsilon-1} \left( \kappa_n (\log T_\epsilon)^{4+4\epsilon_\lambda} T_\epsilon^2 + 1 \right) \\
        &\le Q_w + 2\kappa_n (\log T_\epsilon)^{4+4\epsilon_\lambda} T_\epsilon^3 + 2T_\epsilon,
    \end{split}
    \end{equation}
    where $Q_w$ is the work incurred during the initial $K_w$ iterations. 
    Using condition (b), we have
    \[
    T_\epsilon^3 
    \le M_T^3 \epsilon^{-9/2} (\log(1/\epsilon))^{3+3\delta}(\log\log(1/\epsilon))^3.
    \]
    Moreover, from condition (b), it follows that $\log T_\epsilon = \mathcal{O}(\log(1/\epsilon))$, and hence
    \[
    (\log T_\epsilon)^{4+4\epsilon_\lambda}
    = \mathcal{O}\big((\log(1/\epsilon))^{4+4\epsilon_\lambda}\big).
    \]
    Substituting these bounds into \eqref{eq:wepsilon-sum}, we obtain
    \begin{align*}
    W_\epsilon
    &\le Q_w + 2\kappa_n (\log T_\epsilon)^{4+4\epsilon_\lambda} T_\epsilon^3 + 2T_\epsilon \\
    &\le Q_w 
    + M_W \epsilon^{-9/2}
    (\log(1/\epsilon))^{7+4\epsilon_\lambda+3\delta}
    (\log\log(1/\epsilon))^3
    + 2T_\epsilon,
    \end{align*}
    for some finite random variable $M_W>0$. Since $T_\epsilon$ is of lower order, after enlarging $M_W$ if necessary, we conclude that, for all sufficiently small $\epsilon$,
    \[
    W_\epsilon
    \le
    M_W \epsilon^{-9/2}
    (\log(1/\epsilon))^{\tilde \delta}
    (\log\log(1/\epsilon))^3.
    \]
    This completes the proof.
\end{proof}

\begin{remark}
    The sample complexity of $\mcOtilde(\epsilon^{-4.5})$ for \texttt{Reg-ASTRO} in the absence of any additional structural properties in the objective function or sample paths improves an sample complexity of $\mcOtilde(\epsilon^{-6})$ for \texttt{ASTRO}.
\end{remark}

\section{Complexity Analysis with CRN}\label{sec:crn}

In this section, we investigate the sample complexity by exploiting Common Random Numbers (CRNs). Utilizing CRNs implies that the function evaluations at both the current point $\BFX_k$ and the candidate point $\BFX_k^s$ are performed using an identical sequence of random seeds and the same sample size $N_k$. A primary advantage of this approach is its ability to cancel out common noise components between the two points, provided there is sufficient regularity in the sample functions over the decision space. Consequently, CRNs can lead to a reduced order of the function estimation error relative to the gradient norm, significantly improving the overall sample complexity. To formalize this, we first introduce the following assumption on the sample path.

\begin{assumption}[Sample-path Differentiability] \label{assum:sample-path-c1}
    For every $\xi  \in \Omega\backslash \Omega_0$ where $\Omega_0$ is the largest set of measure zero, the sample function $F(\cdot, \xi)$ is continuously differentiable on $\mathbb{R}^d$.
\end{assumption}
The assumption that the sample paths $F(\cdot,\xi)$ are continuously differentiable with probability one holds in many common stochastic optimization settings where the randomness enters the function in a smooth manner. This is typical when $F$ is constructed from smooth algebraic operations (e.g., sums, products, compositions) involving smooth deterministic functions and random parameters.
More generally, the assumption holds whenever 
the randomness affects only coefficients or inputs rather than the functional form itself, which covers a large class of Monte Carlo–based stochastic optimization problems. 

To quantify the improvement in the scaling of the estimation errors enabled by CRNs, we extend the bounded-error analysis of Section~\ref{sec:complexity} to the CRN setting.

\begin{lemma}
\label{lem:bounded-error-crn}
Suppose Assumptions~\ref{assum:martingale}, \ref{assum:grad-dominates}, and~\ref{assum:sample-path-c1} hold. For iterates $\{\BFX_k\}$ generated by Algorithm~\ref{alg:ASUTRO-CRN} and any constants $\kappa_{eg}, \kappa_{ef} > 0$, the following hold:
\begin{subequations}
\begin{align}
    \mathbb{P}\!\left( \|\BFEbar_k^{g,0}\|
    > \kappa_{eg} \frac{\|\nabla f_k^{0}\|}{\Lambda_k} \ \text{i.o.}\right) &= 0, \label{eq:bounded-error-grad-crn} \\
    \mathbb{P}\!\left(\bigl|\Ebar_k^{0} - \Ebar_k^{s}\bigr|
    > \kappa_{ef} \left(\frac{\|\nabla f_k^{0}\|}{\Lambda_k}\right)^{3/2} \ \text{i.o.}\right) &= 0.  \label{eq:bounded-error-func-crn}
\end{align}
\end{subequations}
\end{lemma}

\begin{proof}
Let $\omega \in \Omega$ denote a realization of Algorithm~\ref{alg:ASUTRO-CRN} with a nonzero probability; we omit $\omega$ for simplicity. It follows from Lemma~\ref{lem:gradient-asfinite} (in light of Remark~\ref{rmk:alpha_error} for the case $\alpha \in [0,2]$) that
\[
\mbP\left(\|\BFEbar_k^{g,0}\| > \kappa_{eg} \max \left\{\Deltapre, \sqrt{\frac{\varepsilon_k}{\Lambda_k}}\right\}^{2} \text{ i.o.}\right) = 0
\]
Under the adaptive sampling rule~\eqref{eq:as-crn}, the bound above
directly implies~\eqref{eq:bounded-error-grad-crn} by the same argument
as in the proof of Lemma~\ref{lem:bounded-error} with $\alpha=2$. 

We now focus on the function error difference in~\eqref{eq:bounded-error-func-crn}. By Assumption~\ref{assum:sample-path-c1} and the use of CRNs, we have
\begin{equation} \label{eq:taylor-error-final}
    \Ebar_k^{s}(N_k) - \Ebar_k^{0}(N_k) = \BFS_k^\top \BFEbar_k^{g,0}(N_k) + \int_{0}^1 \left( \underbrace{\BFEbar^g(\BFX_k + t\BFS_k, N_k) - \BFEbar_k^{g,0}(N_k)}_{\BFDbar_k^g} \right)^\top \BFS_k \, \mathrm{d}t,
\end{equation}
where $\BFEbar^g(\cdot, N_k) := \BFGbar(\cdot, N_k) - \nabla f(\cdot)$.
For all sufficiently large $k$, it follows from~\eqref{eq:bounded-error-grad-crn} that $\|\BFGbar_k^0\| \le (1+\kappa_{eg}/\Lambda_{\min})\|\nabla f_k^0\|$. Recalling $\Delta_k^2 = \|\BFGbar_k^0\|/(16\Lambda_k)$ and $\Lambda_k \ge \Lambda_{\min}$, we obtain
\begin{equation} \label{eq:s_ub_crn}
\|\BFS_k\| \le \Delta_k \le \kappa_\Delta \sqrt{\frac{\|\nabla f_k^0\|}{\Lambda_k}}, \text{ where } \kappa_\Delta := \frac{1}{4}\sqrt{1+\frac{\kappa_{eg}}{\Lambda_{\min}}}.
\end{equation}

For the first term in~\eqref{eq:taylor-error-final}, by~\eqref{eq:s_ub_crn} and 
the bound $\|\BFEbar_k^{g,0}\| \le \kappa_{eg} \Lambda_k^{-1} \|\nabla f_k^0\|$ 
for all sufficiently large $k$ (cf.~\eqref{eq:bounded-error-grad-crn}), we have
\[
|\BFS_k^\top \BFEbar_k^{g,0}|
\le \|\BFS_k\| \, \|\BFEbar_k^{g,0}\|
\le \kappa_\Delta \kappa_{eg}
\left(\frac{\|\nabla f_k^0\|}{\Lambda_k} \right)^{3/2}.
\]

For the integral term in~\eqref{eq:taylor-error-final}, we observe that for any $t \in [0,1]$, the term $\BFEbar^g(\BFX_k + t\BFS_k, N_k)$ represents the gradient estimation error at an intermediate point. By an argument analogous to the proof of~\eqref{eq:bounded-error-grad-crn}, we have for sufficiently large $k$,
\[
\|\BFEbar^g(\BFX_k + t\BFS_k, N_k)\| \le \kappa_{eg} \frac{\|\nabla f_k^0\|}{\Lambda_k} \text{ for all } t \in [0,1].
\]
This holds by Assumption~\ref{assum:varcont} and Remark~\ref{rmk:alpha_error}, since $N_k$ is of order $\Deltatilde_k^{-4}$, yielding uniform control across design points.
Consequently, the integral term is bounded as follows:
\begin{align*}
    \left| \int_{0}^1 (\BFDbar_k^g)^\top \BFS_k \, \mathrm{d}t \right| &\le \int_{0}^1 \left( \|\BFEbar^g(\BFX_k + t\BFS_k, N_k)\| + \|\BFEbar_k^{g,0}\| \right) \|\BFS_k\| \, \mathrm{d}t \\
    &\le 2\kappa_{eg} \frac{\|\nabla f_k^0\|}{\Lambda_k}  \Delta_k \le 2\kappa_{eg} \kappa_\Delta \left(\frac{\|\nabla f_k^0\|}{\Lambda_k}\right)^{3/2}.
\end{align*}
Combining the bounds for both terms in \eqref{eq:taylor-error-final} and defining $\kappa_{ef} := 3\kappa_{\Delta} \kappa_{eg}$, we conclude that the function error difference satisfies \eqref{eq:bounded-error-func-crn}.
\end{proof}

\begin{theorem} \label{thm:bounded-error-crn}
Suppose Assumptions~\ref{assum:martingale}, \ref{assum:grad-dominates}, and~\ref{assum:sample-path-c1} hold. For iterates $\{\BFX_k\}$ generated by Algorithm~\ref{alg:ASUTRO-CRN}, where the adaptive sample sizes for $\BFX_k$ and $\BFX_k^\s$ are chosen according to~\eqref{eq:as-crn}, and any constants $\kappa_{eg}, \kappa_{ef} > 0$, the following hold for any $i \in \{0,\text{s}\}$:
\begin{align*}
    \mbP \left( \|\BFEbar_k^{g,i}\| > \tilde\kappa_{eg} \Delta_k^2 \text{ i.o.} \right) &= 0, \\ 
    \mbP \left(\bigl|\Ebar_k^{0} - \Ebar_k^{s}\bigr| > \tilde\kappa_{ef} \Delta_k^3 \text{ i.o.} \right) &= 0,
\end{align*}
where $\tilde\kappa_{eg} \le \kappa_{eg}(16/(1-\kappa_{eg}\Lambda_{\min}))$ and $\tilde\kappa_{ef} \le \kappa_{ef}(16/(1-\kappa_{eg}\Lambda_{\min}))^{3/2}$.
\end{theorem}

\begin{proof}
    The result follows from the same argument as in the proof of 
    Theorem~\ref{thm:bounded-error}, with Lemma~\ref{lem:bounded-error} replaced by Lemma~\ref{lem:bounded-error-crn}.
\end{proof}

\begin{corollary}\label{cor:hessian-error-crn}
Suppose Assumptions~\ref{assum:smooth-f}--\ref{assum:grad-dominates}, and~\ref{assum:sample-path-c1} hold. Let $\{\BFX_k\}$ be the sequence of iterates generated by Algorithm~\ref{alg:ASUTRO-CRN}. Then, the forward-difference Hessian $\sfH_k$ satisfies
\begin{equation*}
    \mathbb{P}\left(\|\nabla^{2}f(\BFX_k) - \sfH_k\| > \kappa_{e\sfH} \Delta_k \text{ i.o.}\right) = 0,
\end{equation*}
where $\kappa_{e\sfH} := \sqrt{d} \left( \frac{\kappa_{L\sfH}}{2} + \tilde\kappa_{eg} \right)$ and $\tilde{\kappa}_{eg}$ is defined as in Theorem~\ref{thm:bounded-error-crn}.
\end{corollary}

\begin{proof}
    The result follows from the same argument as in the proof of Corollary~\ref{cor:hessian-error}, with Lemma~\ref{lem:bounded-error-crn}.
\end{proof}

The advantage of CRN is only obtained under the following somewhat strong assumption on the model Hessian.

\begin{assumption} \label{assum:Hessian_bound_delta}
    There exists $\tau \in (0,1]$ such that, for every $k \in \mbN$, $\|\sfH_k\| \le \tfrac{1}{\tau}\tfrac{\|\BFGbar_k\|}{\Delta_k}$.
\end{assumption} 

A condition of the same form is imposed in the convergence analysis of the stochastic trust-region method in~\cite{rinaldi2024stochastic}. Under this condition, every trust-region step is shown to satisfy the uniform lower bound
\(
    \|\BFS_k\| \ge \tau\Delta_k,
\)
which can replace~\eqref{eq:stepsize_lb} in the corresponding convergence analysis. We start by proving the same kind of uniform lower bound for the step-size.

\begin{lemma} \label{lem:stepsize_lb_delta}
    Suppose Assumption~\ref{assum:Hessian_bound_delta} holds. Then there exists a constant $\tilde\tau > 0$ such that $\|\BFS_k\| \ge \tilde\tau \Delta_k$. 
\end{lemma}
\begin{proof}
    We know from~\eqref{eq:Sk_lowerbound} and Assumption~\ref{assum:Hessian_bound_delta} that 
    \[
    \| \BFS_k \| \ge \frac{\|\BFGbar_k^0\|}{\|\sfH_k\| + \sqrt{\Lambda_k \|\BFGbar_k^0\|} } 
    \ge \frac{16 \Lambda_k \Delta_k^2}{ \frac{1}{\tau}\frac{\|\BFGbar_k\|}{\Delta_k} + 4 \Lambda_k \Delta_k} 
    = \frac{16 \Lambda_k \Delta_k^2}{ \frac{16 \Lambda_k}{\tau}\Delta_k + 4 \Lambda_k \Delta_k} 
    = \frac{16 \tau}{16 + 4 \tau}\Delta_k,
    \]
    which completes the proof with $\tilde\tau = 16\tau(16+4\tau)^{-1}$.
\end{proof}

Indeed, under Lemma~\ref{lem:stepsize_lb_delta}, the gradient contraction condition for successful iterations is no longer required in Algorithm~\ref{alg:ASUTRO} to establish the iteration complexity $\mcOtilde(\epsilon^{-3/2})$. The reason is that, in combination with Lemma~\ref{lem:stepsize_lb_delta}, the model reduction bound in Lemma~\ref{lem:model-reduction} becomes of order $\mcO(\Delta_k^3)$. As a result, the standard sufficient reduction condition in trust-region methods guarantees
\(
    \Fbar_k^0 - \Fbar_k^s \ge \mcO(\Delta_k^3)
\)
for every successful iteration. Therefore, the gradient contraction condition can be dropped without affecting the iteration complexity result, and we summarize the resulting~\texttt{Reg-ASTRO} with in Algorithm~\ref{alg:ASUTRO-CRN}.

\begin{algorithm}[htp]  \scriptsize 
\caption{\texttt{Reg-ASTRO} with CRN}
\label{alg:ASUTRO-CRN}
\begin{algorithmic}[1]
\Require $\BFx_{0}, \Delta_{0}^\pre,\Delta_{\max}^\pre,  \Lambda_{\min}, \sigma_0, \eta\in(\frac{1}{4},1), \gamma_1>1>\gamma_2>0, \lambda_k = \mcO((\log k)^{1+\epsilon_\lambda})>1$.

\State Estimate $\BFGbar_0^\pre$ using $\Delta_0^\pre$. Set $\Lambda_0=\max\left(\Lambda_{\min},\frac{\|\BFGbar_0^\pre\|}{16(\Delta_0^\pre)^2}\right)$ and $k=0$. \label{algstep:first-estimate-parameter}
\For{$k=0,1,2,\cdots$}
\State Evaluate $\Fbar_k=\Fbar(\BFX_k,N_k),\ \BFGbar_k=\BFGbar(\BFX_k,N_k)$ using \eqref{eq:sig-mx} and \eqref{eq:delta-tilde} with sample size
\begin{equation}
    N_k=\min\left\{n:\frac{\sigma_{\text{mx}}(\BFX_k,n)}{\sqrt{n}}\leq \frac{\kappa_a}{\sqrt{\lambda_k}} \Deltatilde_k^2(n)\right\}.\label{eq:as-crn}
\end{equation} \label{algstep:evaluate}
\State Find an approximate Hessian $\sfH_k\in\real^{d\times d}$ to construct the model as in \eqref{eq:model}.

\State Set $\Delta_k = \sqrt{\frac{\|\BFGbar_k\|}{16 \Lambda_k}}$ and the trial point $\BFX_k^\s=\BFX_k+\BFS_k$ by solving \eqref{eq:subproblem}. 
 \label{ASUTRO:subproblem}

\State Evaluate $\Fbar_k^\s=\Fbar(\BFX_k^\s,N_k)$ and $ \BFGbar_k^\s=\BFGbar(\BFX_k^\s,N_k)$ and compute $\rho_k$ as in \eqref{eq:success-ratio}.

\If{$\rho_k > \eta$} 
\label{ASUTRO:successful-iteration-condition}
\State Update $\BFX_{k+1} = \BFX_k^s$, $\Lambda_{k+1} = \max\left\{\gamma_2 \Lambda_k, \Lambda_{\min}\right\}$, and  $\BFGbar_{k+1}^\pre=\BFGbar_{k}^\s$. [Successful]

\Else
\State Update $\BFX_{k+1} = \BFX_k$, $\Lambda_{k+1} = \gamma_1 \Lambda_k$, and $\BFGbar_{k+1}^\pre=\BFGbar_{k}$. [Unsuccessful]

\EndIf

\State Set $\Delta_{k+1}^\pre = \min\left\{\sqrt{\frac{\|\BFGbar_{k+1}^\pre\|}{16\Lambda_{k+1}}},\Delta_{\max}^\pre\right\}$ and $k = k+1$.
\label{algstep:update-delta}
\EndFor
\end{algorithmic}
\end{algorithm}

We now begin to establish the complexity results. As a first step, we present the CRN counterparts of Lemmas~\ref{lem:reg-ub} and~\ref{lem:monotone_decreasing} from Section~4, which establish a uniform upper bound on $\Lambda_k$ and a sufficient decrease in true function on successful iterations.

\begin{lemma} \label{lem:reg-ub-crn}
    Suppose that Assumptions~\ref{assum:smooth-f}-\ref{assum:grad-dominates} and~\ref{assum:sample-path-c1}-\ref{assum:Hessian_bound_delta} hold. For iterates $\{\BFX_k\}$ generated by Algorithm~\ref{alg:ASUTRO}, there exists a constant $\Lambda_{\max}$ such that
    \begin{equation}
        \mbP\left( \Lambda_k > \gamma_1 \Lambda_{\max} \text{ i.o.}\right) = 0.
    \end{equation}
\end{lemma}
\begin{proof}
    Let $\omega \in \Omega$ denote a realization of Algorithm~\ref{alg:ASUTRO} with a nonzero probability. For the remainder of this proof, we omit the $\omega$ for simplicity.
    Following the same argument in the proof of Lemma~\ref{lem:reg-ub} (Case 1) with Theorem~\ref{thm:bounded-error-crn} and Corollary~\ref{cor:hessian-error-crn}, we obtain, for all sufficiently large $k$,
    \begin{equation} \label{eq:model-error-candidate}
    \begin{split}
        |\Fbar_k^\s - M_k(\BFS_k)|
        & \le \left( \tilde\kappa_{eg} + \frac{\kappa_{e\sfH}}{2} + \frac{\kappa_{L\sfH}}{6} + \tilde\kappa_{ef} \right) \Delta_k^3.
    \end{split}    
    \end{equation}
    From Lemma~\ref{lem:model-reduction} and~\ref{lem:stepsize_lb_delta}, the success ratio satisfies
    \begin{align} \label{eq:success-ratio-ub-crn}
        |1-\hat{\rho}_k| &= \frac{|\Fbar_k^\s - M_k(\BFS_k)|}{|M_k(\boldsymbol{0})-M_k(\BFS_k)|} 
        \le \frac{\left( \tilde\kappa_{eg} + \frac{\kappa_{e\sfH}}{2} + \frac{\kappa_{L\sfH}}{6} + \tilde\kappa_{ef} \right) \Delta_k^3}{ 2\tilde\tau^2\Lambda_k  \Delta_k^3} \nonumber\\
        &\le \frac{\left( \tilde\kappa_{eg} + \frac{\kappa_{e\sfH}}{2} + \frac{\kappa_{L\sfH}}{6} + \tilde\kappa_{ef} \right)}{2\tilde\tau^2\Lambda_{\max}}.
    \end{align} 
    When $\Lambda_{\max} = \left( \tilde\kappa_{eg} + \frac{\kappa_{e\sfH}}{2} + \frac{\kappa_{L\sfH}}{6} + \tilde\kappa_{ef} \right)(2\tilde\tau^2(1-\eta))^{-1}$, we ensure that $|1-\rho_k| \le 1-\eta$, which implies $\rho_k \ge \eta$. This completes the proof.
\end{proof}

\begin{lemma}\label{lem:monotone_decreasing-crn}
    Suppose that Assumptions~\ref{assum:smooth-f}-\ref{assum:grad-dominates} and~\ref{assum:sample-path-c1}-\ref{assum:Hessian_bound_delta} hold. Define the index set of the successful iterations by $\mcS = \{k \in \mbN: \rho_k > \eta \}$. For iterates $\{\BFX_k\}$ generated by Algorithm~\ref{alg:ASUTRO-CRN}, there exists a finite positive constant $\kappa_{s}$ such that
    \begin{align*}
    \mbP\!\left( \bigl(k\in \mcS\bigr) \bigcap \bigl(f_k^0 - f_k^{\s} < \kappa_{s} \|\nabla f_k^0\|^{3/2} \bigr)\ \text{i.o.} \right) &= 0,
    \end{align*}
\end{lemma}
\begin{proof}
    Let $\omega \in \Omega$ denote a realization of Algorithm~\ref{alg:ASUTRO-CRN} with a nonzero probability. For the remainder of this proof, we omit the $\omega$ for simplicity.
    By Lemma~\ref{lem:model-reduction}, the decrease in the function value can be written as
    \begin{equation*}
    \begin{split}    
        2 \tilde\tau^2 \eta \Lambda_k \Delta_k^3 \le \Fbar_k^0 - \Fbar_k^\s = f_k^0 - f_k^\s + \Ebar_k^{0} - \Ebar_k^{\text{s}} 
    \end{split}
    \end{equation*}
    By Theorem~\ref{thm:bounded-error-crn}, we obtain, for all sufficiently large $k$,
    \[
    f_k^0 - f_k^\s \ge 2\tilde\tau^2 \eta \Lambda_{k} \Delta_k^3 -  |\Ebar_k^{0}-\Ebar_k^{\text{s}}| \ge 2\tilde\tau^2 \eta \Lambda_{k} \Delta_k^3 - \tilde\kappa_{ef}\Delta_k^3.
    \]
    By choosing $\kappa_{eg}\le (2\Lambda_{\min})^{-1}$ and $\kappa_{ef} \le \bar\tau^2\eta\Lambda_{\min}/(128\sqrt{2})$ in Theorem~\ref{thm:bounded-error-crn}, we get $\tilde\kappa_{ef} \le \bar\tau^2\eta\Lambda_{\min}$ and consequently obtain for sufficiently large $k$, $f_k^0 - f_k^\s \ge \tilde\tau^2 \eta \Lambda_{k} \Delta_k^3$. Using $\Lambda_k \le \gamma_1 \Lambda_{\max}$ (Lemma~\ref{lem:reg-ub-crn}) and $\Delta_k = (\|\BFGbar_k^0\|/(16 \Lambda_k))^{1/2}$, we obtain
    \[
    f_k^0 - f_k^\s 
    \ge \tilde\tau^2 \eta \Lambda_k \left( \frac{\|\BFGbar_k^0\|}{16 \Lambda_k} \right)^{3/2}
    = \frac{\tilde\tau^2 \eta}{64} \frac{\|\BFGbar_k^0\|^{3/2}}{\sqrt{\Lambda_k}}
    \ge \frac{\tilde\tau^2 \eta}{64\sqrt{\gamma_1 \Lambda_{\max}}}
    \|\BFGbar_k^0\|^{3/2}.
    \]
    Moreover, by Lemma~\ref{lem:bounded-error-crn}, we have, for sufficiently large $k$,
    \(
    \|\BFGbar_k^0 - \nabla f_k^0\| \le \frac{\kappa_{eg}}{\Lambda_{\min}} \|\nabla f_k^0\|,
    \)
    which implies
    \[
    \|\BFGbar_k^0\|
    \ge \|\nabla f_k^0\| - \|\BFGbar_k^0 - \nabla f_k^0\|
    \ge \left(1 - \frac{\kappa_{eg}}{\Lambda_{\min}}\right)\|\nabla f_k^0\|.
    \]
    Therefore, we obtain, for sufficiently large $k$,
    \[
    f_k^0 - f_k^\s 
    \ge \underbrace{\frac{\tilde\tau^2 \eta}{64\sqrt{\gamma_1 \Lambda_{\max}}}
    \left(1 - \frac{\kappa_{eg}}{\Lambda_{\min}}\right)^{3/2}}_{:=\kappa_s}
    \|\nabla f_k^0\|^{3/2}.
    \]
    By choosing $\kappa_{eg} \le \min\{(2\Lambda_{\min})^{-1},\Lambda_{\min}\}$, we obtain $\kappa_{\s} > 0$, which completes the proof.
\end{proof}

\begin{remark}
The exact solution of the subproblem in Lemma~\ref{lem:exactsln} is not essential for the analysis in this section. It is sufficient that the trial step $\BFS_k$ satisfies the two conditions
\begin{equation*}
\| \BFGbar_k^0 + (\sfH_k + c\sqrt{\Lambda_k \|\BFGbar_k^0\|}\sfI_d)\BFS_k \|
\le
c\sqrt{\Lambda_k\|\BFGbar_k^0\|}\,\|\BFS_k\|,
\end{equation*}
and
\begin{equation*}
\BFS_k^T\left(\sfH_k + c\sqrt{\Lambda_k\|\BFGbar_k^0\|}\sfI_d\right)\BFS_k > 0,
\end{equation*}
for some $c \in [0,1)$.
Combining these conditions with Lemma~\ref{lem:stepsize_lb_delta} yields a model reduction of order at least $\Delta_k^3$, which plays two key roles in the analysis. First, it provides a lower bound for the denominator of the success ratio~\eqref{eq:success-ratio-ub-crn}. As a result, there exists a finite constant $\tilde \Lambda_{\max} > 0$ such that any iteration with $\Lambda_k \ge \tilde \Lambda_{\max}$ must be successful for sufficiently large $k$, thereby extending Lemma~\ref{lem:reg-ub-crn} to this inexact setting. Second, it ensures that the reduction in the true objective is of order at least $\|\nabla f_k^0\|^{3/2}$, which aligns with the argument in Lemma~\ref{lem:monotone_decreasing-crn}. Therefore, all subsequent arguments, including the iteration and sample complexity results, continue to hold under these inexact subproblem conditions without requiring an exact solution.
\end{remark}

We now establish the main complexity results for the CRN setting. 
In particular, we first derive an iteration-complexity bound for reaching an $\epsilon$-first-order stationary point, 
and then characterize the corresponding sample complexity.

\begin{theorem}[Iteration Complexity]\label{thm:ic-crn} 
    Suppose that Assumptions~\ref{assum:smooth-f}-\ref{assum:grad-dominates} and~\ref{assum:sample-path-c1}-\ref{assum:Hessian_bound_delta} hold.
    There exists a positive random variable $M_T > 0$ such that for any deterministic sequence $\{\epsilon_{\ell}\}$ with $\epsilon_{\ell} \downarrow 0$:
    \[
    \mbP\!\left(T_{\epsilon_{\ell}}>M_{T}\epsilon_{\ell}^{-3/2}\ \text{i.o.}\right)=0,
    \]
    where $T_{\epsilon_{\ell}} := \min \{k : \|\nabla f_k^0\| \le \epsilon_{\ell}\}$.
\end{theorem}
\begin{proof}
    Let $\omega \in \Omega$ denote a realization of Algorithm~\ref{alg:ASUTRO-CRN} with a nonzero probability; we omit the $\omega$ for simplicity.
    There exists an index $K$ such that for all $k\ge K$:
    \begin{itemize}
        \item[(i)] By Lemma~\ref{lem:monotone_decreasing-crn}, 
        \(f_k^0 - f_k^\s \ge \kappa_s \|\nabla f_k^0\|^{3/2}\);
        \item[(ii)] By Lemma~\ref{lem:reg-ub-crn}, $\Lambda_k \le \gamma_1 \Lambda_{\max}$.
    \end{itemize}
    Fix $\epsilon_0>0$ such that $\epsilon_0^{3/2}<1$ and $K<T_\epsilon$ for all $\epsilon\le\epsilon_0$. Let $\epsilon\in(0,\epsilon_0]$ and set 
    \[
    q_{\mcU}:=\left\lceil \frac{\log(\Lambda_{\max}/\Lambda_{\min})}{\log \gamma_1}\right\rceil.
    \]
    Between any two consecutive successful indices in $[K,T_\epsilon)$ there are at most $q_{\mcU}$ unsuccessful iterations; otherwise, after $q_{\mcU}$ consecutive unsuccessful iterations we would have $\Lambda_k > \gamma_1 \Lambda_{\max}$ (since each unsuccessful iteration multiplies $\Lambda_k$ by $\gamma_1$), which contradicts (ii). Consequently,
    \(
    T_\epsilon \le K + (q_{\mcU}+1)\,|\,[K,T_\epsilon)\cap\mcS\,|.
    \)
    Furthermore, for $k\in[K,T_\epsilon)\cap\mcS$ we have $\|\nabla f_k^0\|>\epsilon$ and $f_k^0-f_k^\s\ge \kappa_s\|\nabla f_k^0\|^{3/2}\ge \kappa_s\epsilon^{3/2}$. Hence, we obtain 
    \[
    f_K^0-f^* \ge \sum_{k\in[K,T_\epsilon)\cap\mcS}(f_k^0-f_k^\s)
    \ge \kappa_s\epsilon^{3/2}\,|\,[K,T_\epsilon)\cap\mcS\,|.
    \]
    Combining the above bounds yields, for all $\epsilon \in (0,\epsilon_0]$,
    \[
    T_\epsilon\,\epsilon^{3/2}
    \le K\epsilon^{3/2} + (q_{\mcU}+1)\frac{f_K^0 - f^*}{\kappa_s}
    \le K\epsilon_0^{3/2} + (q_{\mcU}+1)\frac{f_K^0 - f^*}{\kappa_s}
    =: M_T.
    \]
    This completes the proof.
\end{proof}

\begin{theorem}[Sample Complexity]\label{thm:sample_complexity_CRN}
    Suppose that Assumptions~\ref{assum:smooth-f}-\ref{assum:grad-dominates} and~\ref{assum:sample-path-c1}-\ref{assum:Hessian_bound_delta} hold.
    There exists a positive random variable $M_W > 0$ such that for any deterministic sequence $\{\epsilon_{\ell}\}$ with $\epsilon_{\ell} \downarrow 0$:
    \[
    \mbP\!\left( W_{\epsilon_{\ell}}  > M_{W}\epsilon_{\ell}^{-7/2}(\log(1/\epsilon_{\ell}))^{-2}  \ \text{i.o.}\right) = 0,
    \]
    where $W_{\epsilon_\ell}:= 2\sum_{k=0}^{T_{\epsilon_{\ell}}-1} N_k$ is the total evaluations of $F(\cdot,\cdot)$ and $\BFG(\cdot,\cdot)$.
\end{theorem}

\begin{proof}
    Let $\omega \in \Omega$ denote a realization of Algorithm~\ref{alg:ASUTRO-CRN} with a nonzero probability; we omit the $\omega$ for simplicity. By following the same argument in the proof of Theorem~\ref{thm:W-eps-rate}, there exists an index $K_{fg}$ such that, for all $k \ge K_{fg}$, 
    \[
        N_k \le \left\lceil \frac{\bar\sigma^{2}\,\lambda_k\,\Lambda_k^{2}}{\kappa_a^{2} \varepsilon_k^2} \right\rceil, \text{ where } \bar\sigma^2:=\max\{\sigma_0^2, 2\sigma_F^2, 2d \sigma_g^2\}.
    \]
    By the argument in the proof of Lemma~\ref{lem:reg-ub-crn}, there exists an index $K_\lambda$ such that for all $k \ge K_{\lambda}$, $\Lambda_k \le \gamma_1\Lambda_{\max}$.
    Fix a sufficiently small $\epsilon_0 > 0$ such that for all $\epsilon < \epsilon_0$:
    \begin{itemize}
        \item[(a)] $T_\epsilon > K_w := \max\{K_{fg},K_{\lambda}\}$;
        \item[(b)] $T_\epsilon \le M_T \epsilon^{-3/2}$ by Theorem~\ref{thm:ic-crn}.
    \end{itemize}
    For all $k \in (K_w, T_\epsilon]$, using $\varepsilon_k = c_* k^{-2/3}$ with $c_* > 0$, the number of samples at each iteration is bounded by
    \[
    N_k \le \left\lceil \frac{\bar\sigma^{2} \gamma_1^2 \Lambda_{\max}^{2}}{\kappa_a^{2}}\lambda_k\varepsilon_k^{-2} \right\rceil 
    \le \left\lceil \frac{\bar\sigma^{2} \gamma_1^2 \Lambda_{\max}^{2}}{c_*^2\kappa_a^2} \lambda_k k^{4/3} \right\rceil 
    \le \left\lceil \kappa_{n} (\log T_\epsilon)^{1+\epsilon_\lambda} T_\epsilon^{4/3} \right\rceil,
    \]
    where $\kappa_n := (\bar\sigma^{2} \gamma_1^2 \Lambda_{\max}^{2})(c_*\kappa_a)^{-2}$.
    Consequently,
    \begin{equation}
    \label{eq:wepsilon-ub-crn}
    \begin{split}    
    W_\epsilon
    = 2\sum_{k=0}^{T_\epsilon-1}N_k
    & \le 2\sum_{k=0}^{K_w-1} N_k  
    + 2 \sum_{k=K_w}^{T_\epsilon-1} \left( \kappa_n (\log T_\epsilon)^{1+\epsilon_\lambda} T_\epsilon^{4/3} + 1 \right)\\
    & \le Q_w + 2 \kappa_n (\log T_\epsilon)^{1+\epsilon_\lambda} T_\epsilon^{7/3} + 2 T_\epsilon
    \end{split}
    \end{equation}
    where $Q_w$ accounts for the initial sum.
    By condition~(b), we know that $T_\epsilon \le M_T \epsilon^{-3/2}$. Substituting this into~\eqref{eq:wepsilon-ub-crn}, we obtain, for some finite random variable $M_W > 0$,
    \[
    W_\epsilon \le Q_w + 2 \kappa_n \left(\log (M_T \epsilon^{-3/2})\right)^{2} (M_T \epsilon^{-3/2})^{7/3} + 2(M_T \epsilon^{-3/2})
    \le M_{W} (\log(1/\epsilon))^2 \epsilon^{-7/2},
    \]
    where the last inequality follows from the fact that the lower-order terms and the constant offsets are absorbed by the dominant $\epsilon^{-7/2}$ factor as $\epsilon \downarrow 0$. This completes the proof.
\end{proof}

\section{Numerical Experiments} 
We will now compare \texttt{Reg-ASTRO} with other solvers such as \texttt{ADAM}~\cite{kingma2014adam} and \texttt{ASTRO}~\cite{ha2025complexity}. The evaluation of the solvers will be conducted using SimOpt~\cite{Eckman_SimOpt}. 

\subsection{Stochastic Rosenbrock Problem}
We begin with the Rosenbrock function
\[
f(\BFx) = \sum_{i=1}^{d-1} 10(x_{i+1} - x_i^2)^2 + (1 - x_i)^2.
\]
The gradient of the deterministic component \( f(\BFx) \) is given analytically as
\[
\begin{aligned}
\frac{\partial f}{\partial x_1} &= -40x_1(x_2 - x_1^2) - 2(1 - x_1), \\
\frac{\partial f}{\partial x_i} &= 20(x_i - x_{i-1}^2) - 40x_i(x_{i+1} - x_i^2) - 2(1 - x_i), \quad i = 2, \ldots, d-1, \\
\frac{\partial f}{\partial x_d} &= 20(x_d - x_{d-1}^2).
\end{aligned}
\]
The Rosenbrock function provides a smooth yet nonconvex landscape with a narrow curved valley, posing a well-known challenge for gradient-based methods. 
This structure makes it particularly suitable for testing \texttt{Reg-ASTRO}, as it enables us to examine how the regularization mechanism and the gradient-contraction-based criterion contribute to finding improved solutions under nonconvex conditions within limited budget. To introduce stochasticity, we perturb both the function and gradient evaluations as
\[
F(\BFx,\bm{\xi}^f) = f(\BFx) + \sum_{i=1}^{d} \xi^f_i, 
\qquad
\BFG(\BFx,\bm{\xi}^g) = \nabla f(\BFx) + \sum_{i=1}^{d} \xi^g_i \BFe_i,
\]
where
\(
\xi^f_i \sim \mathcal{N}(0, 1) 
\) and 
\(
\xi^g_i \sim \mathcal{N}(0, 1).
\)
We evaluate the solvers under two problem dimensions, $d=5$ and $d=30$; see Figure~\ref{fig:rosen}. \texttt{Reg-ASTRO} exhibits a clear advantage over \texttt{ASTRO} in non-convex settings under an identical hyperparameter setting, and demonstrates sufficiently fast convergence behavior. In contrast, \texttt{ADAM} either shows slower initial progress compared to \texttt{Reg-ASTRO} and \texttt{ASTRO}, as observed in Figure~\ref{fig:5d}, or exhibits rapid initial progress followed by stagnation at later stages, as observed in Figure~\ref{fig:30d}, likely due to its use of a fixed sample size.

 
\begin{figure} [htp]
\centering
\subfloat[$d = 5$]{%
\resizebox*{6cm}{!}{\includegraphics{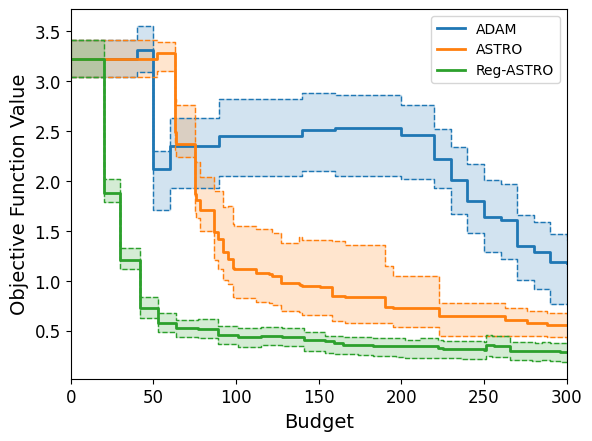}}\label{fig:5d}}
\subfloat[$d = 30$]{%
\resizebox*{6cm}{!}{\includegraphics{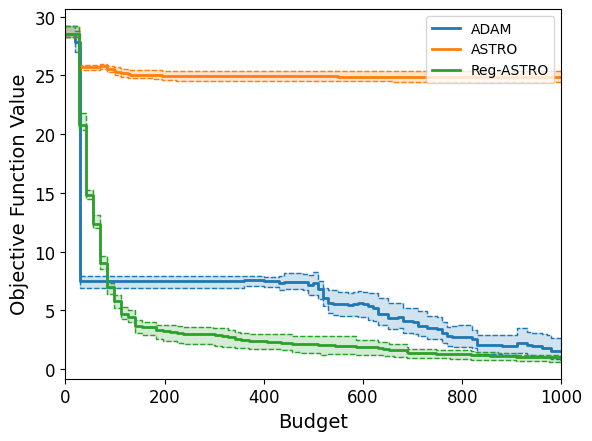}}\label{fig:30d}}
\caption{Mean value and 90\% confidence interval of objective function trajectory per spent budget. \texttt{Reg-ASTRO} consistently attains better solutions in nonconvex landscapes.} 
\label{fig:rosen}
\end{figure}

\subsection{Ambulance Deployment Problem}
For a more realistic and application-driven experiment, we consider an ambulance deployment problem based on a discrete-event simulation model. This problem captures key features of stochastic service systems and provides a challenging nonconvex testbed for simulation optimization, where noisy first-order oracle information is available via the IPA estimator; see Remark~\ref{rmk:autodiff}.

The system models a multi-base ambulance dispatch setting over a two-dimensional service region. Specifically, emergency calls arrive randomly over a square region $[0,20] \times [0,20]$, and each call requires dispatching the nearest available ambulance. If all ambulances are busy, incoming calls are queued until a unit becomes available. Upon dispatch, an ambulance travels to the call location, serves the request for a random duration, and then returns to the pool of available units.

The system consists of two types of ambulance bases: (i) fixed bases located at predetermined coordinates, and (ii) variable bases whose locations are decision variables. In our setting, three fixed bases are placed at $(5,5)$, $(5,15)$, and $(15,15)$, while two additional bases are allowed to move freely within the region. Since each variable base has two coordinates, the resulting optimization problem is four-dimensional. The objective is to minimize the expected response time of the system. The problem is inherently stochastic due to multiple sources of randomness, including: (i) exponentially distributed inter-arrival times of emergency calls, (ii) uniformly distributed spatial locations of calls, and (iii) exponentially distributed on-scene service times. 

We evaluate the performance of \texttt{Reg-ASTRO}, \texttt{ASTRO}, and \texttt{ADAM} under multiple initializations of the variable base locations. Specifically, we consider 20 different initial points, and for each initial point, we perform 20 independent runs for each solver. For each initial point, we define a reference solution as the best solution obtained across all solvers and runs. We then assess the convergence behavior of each method by measuring the optimality gap of its solution trajectory relative to this reference value; see Figure~\ref{fig:ambulance}. \texttt{Reg-ASTRO} achieves faster and more stable convergence compared to the other methods, which is consistent with the theoretical properties established in our analysis.

\begin{figure} [htp]
\centering
\subfloat[Progress curves]{%
\resizebox*{6cm}{!}{\includegraphics{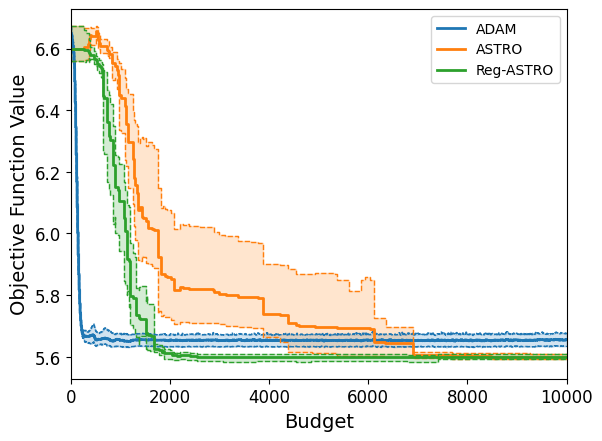}}\label{fig:amb-1}}
\subfloat[Solvability profiles]{%
\resizebox*{6cm}{!}{\includegraphics{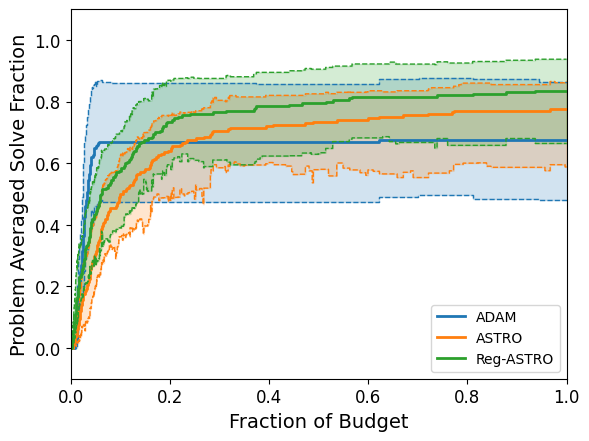}}\label{fig:amb-2}}
\caption{Performance comparison on the ambulance deployment problem. (a) Progress curves on a single instance, showing the mean objective value with 90\% confidence intervals versus budget. (b) Solvability profiles showing the fraction of problems solved within a 5\% optimality gap, with 90\% confidence intervals.} 
\label{fig:ambulance}
\end{figure}

\section{Concluding Remarks}
This paper crucially proposes a simple and close to its deterministic variant algorithm for stochastic trust-region optimization with adaptive sampling. The main changes to the classic methods is the way the trust-region radius is updated, which no longer ensures increase of the step-size after each success or decrease of the step-size otherwise. Similar acceptance criteria to the deterministic regularized trust-region methods in \cite{jiang2026beyond} are used but both are augmented by secondary conditions due to the fact that we use estimated quantities here. Criteria 1 in Algorithm~\ref{alg:ASUTRO} ensures that sufficient reduction leads to success only if the step-size is not too small (which could be shown to automatically hold in the deterministic setting
).  Criteria 2 in Algorithm~\ref{alg:ASUTRO}  ensures that gradient contraction leads to success only if the regularization coefficient is large enough relative to the estimate of the gradient norm. 

We conclude this paper by pointing out the open question of relaxing the  required Hessian assumption in the improved sample complexity of Section~\ref{sec:crn}. Another open direction is to apply similar quadratic regularization strategies for a derivative-free optimization problem. In that setting, the model quality is similarly handled with fully quadratic properties by means of interpolating a set of neighboring points; similar adaptive sampling conditions are also anticipated. However, the difficulty of increased complexity in terms of the dimension of the problem will need to be investigated. 




\bmhead{Acknowledgements}
This research was partially supported by Office of Naval Research Grant N000142412398 and N000142312588, and the National Science Foundation NSF-RTG Grant DMS-2134107.





\bmhead{Conflict of Interest}
The authors declare no conflict of interest.








\begin{appendices}





\end{appendices}


\bibliography{sn-bibliography}

\end{document}